\documentclass[a4paper,9pt]{article}
\usepackage{amsfonts}
\usepackage{amssymb,amsthm,color,tikz}
\usepackage{amsmath}
\usepackage{comment,mathtools,esint}
\usepackage{cancel}

\usepackage[normalem]{ulem}
\usepackage{cancel}

\newcommand{\R}{\mathbb{R}}
\newcommand{\N}{\mathbb{N}}

\newcommand{\E}{\mathbb{E}}

\newtheorem{theorem}{Theorem}
\newtheorem{lemma}[theorem]{Lemma}

\newtheorem{example}{Example}
\newtheorem{definition}{Definition}
\newtheorem{remark}{Remark}

\newcommand{\Om}{\Omega}
\newcommand{\lb}{\lambda}

\newcommand{\sm}{\setminus}

\newcommand{\sq}{\subseteq}
\newcommand{\ov}{\overline}

\newcommand{\vps}{\varepsilon}
\newcommand{\bp}{\begin{proof}}
\newcommand{\ep}{\end{proof}}

\begin{document}
\title{On efficiency and localisation for the torsion function}

\author{{M. van den Berg} \\
School of Mathematics, University of Bristol\\
Fry Building, Woodland Road\\
Bristol BS8 1UG, United Kingdom\\
\texttt{mamvdb@bristol.ac.uk}\\
\\
{D. Bucur}\\
Laboratoire de Math\'ematiques, Universit\'e Savoie Mont Blanc \\
UMR CNRS  5127\\
Campus Scientifique,
73376 Le-Bourget-Du-Lac, France\\
\texttt{dorin.bucur@univ-savoie.fr}\\
\\
{T. Kappeler}\\
Institut f\"ur Mathematik, Universit\"at Z\"urich\\
 Winterthurerstrasse 190,
CH-8057 Z\"urich, Switzerland\\
 \texttt{thomas.kappeler@math.uzh.ch}}
\date{20 April 2021}\maketitle
\vskip 3truecm \indent
\begin{abstract}\noindent
We consider the torsion function for the Dirichlet Laplacian $-\Delta$, and for the Schr\"odinger operator $- \Delta + V$ on an open set $\Omega\subset \R^m$ of finite Lebesgue measure
$0<|\Omega|<\infty$ with a real-valued, non-negative, measurable potential $V.$
We investigate the efficiency and the phenomenon of localisation for the torsion function, and their interplay with the geometry of the first Dirichlet eigenfunction.

\end{abstract}
\medskip \noindent \ \ \ \   { Mathematics Subject
Classification (2000)}: 35J05, 35J25, 35K05, 35P99.
\begin{center} \textbf{Keywords}: Torsion function, first Dirichlet eigenfunction, Schr\"odinger operator, Dirichlet boundary condition, localisation, efficiency.
\end{center}

\section{Introduction and main results \label{sec1}}
Let $\Omega$ be an open set in $\R^m$, with finite Lebesgue measure, $0 < |\Omega| < \infty$, and boundary $\partial\Omega$, and let
\begin{equation*}%\label{e1}
L=-\Delta+V,
\end{equation*}
be the Schr\"odinger operator acting in $L^2(\Omega)$ with the potential  $V:\Omega\rightarrow \R^+,\, \R^+= [0,\infty) $ being measurable.
The torsion function for $\Om$ is the unique solution of
\begin{equation*}%\label{e3}
-\Delta v=1,\, \qquad v\in H_0^{ 1 }(\Omega).
\end{equation*}
It is denoted by $v_{\Om}$, and is also referred to as the torsion function for the Dirichlet Laplacian.
The function $v_{\Omega}$ is non-negative, pointwise increasing in $\Omega$, and satisfies,
\begin{equation}\label{e4}
\lambda_1(\Omega)^{-1}< \|v_{\Omega}\|_{\infty}\le
(4+3m\log 2)\lambda_1(\Omega)^{-1},
\end{equation}
where
\begin{equation*}%\label{e2}
\lambda_1(\Omega)=\inf_{\varphi\in
H_0^1(\Omega)\setminus\{0\}}\frac{\Vert \nabla\varphi\Vert_2^2}{
\Vert \varphi\Vert_2^2},
\end{equation*}
is the first eigenvalue of the Dirichlet Laplacian. Here, and throughout the paper, $\Vert\cdot\Vert_p$ denotes the standard $L^p$ norm, $1\le p\le \infty$.
Since $|\Om|<\infty$ the first eigenvalue is bounded away from $0$ by the Faber-Krahn inequality.
The $m$-dependent constant in the right-hand side of \eqref{e4} has
subsequently been improved (\cite{GS},\cite{HV}). We denote the sharp constant by
\begin{equation}\label{e5}
\mathfrak{c}_m=\sup\{\lambda_1(\Omega)\|v_{\Omega}\|_{\infty}:\Omega\,\textup{ open in }\,\R^m,\,0<|\Omega|<\infty\}.
\end{equation}

More generally, the equation $Lv = 1$ also has a unique solution $v_{\Omega, V}\in H^1_0(\Omega)$, referred to as
the torsion function for $L$.

In this paper we study the {\em efficiency} of the torsion function of Schr\"odinger operators, and study the phenomenon of {\em localisation}.
The notion of efficiency, or mean to max ratio, goes back to \cite{ps} and \cite {sperb}, where it was introduced for the first Dirichlet eigenfunction.
It can be viewed as a (rough) measure of localisation.
The mean to max ratio for the torsion function for bounded, open, convex sets in Euclidean space was studied in \cite{DPGGLB} in the more general context of the $p$-torsional rigidity.
The phenomenon of localisation of eigenfunctions of Schr\"odinger operators is a prominent and very active research area and has important applications in the applied sciences.
The literature is extensive. See for example the review paper \cite{grebenkov}. It was discovered in \cite{DDJM1} and \cite{DDJM2}, and the references therein, that under appropriate conditions, $v_{\Om,V}^{-1}$ can be used for approximating eigenvalues and eigenfunctions of $L$.
It raises the question as to whether under appropriate assumptions, the phenomenon of localisation can also be observed for the torsion function of Schr\"odinger operators,
and suggests to investigate the interplay between the localisation of the torsion function and the one of the first Dirichlet eigenfunction.

The main results of this papers can be described in an informal way as follows.
Theorem \ref{the1} compares the efficiency of $v_{\Omega,V}$ with the one of $v_\Omega$ under a variety of hypotheses.
In addition it shows that for any given $\Omega$,  the efficiency for $v_{\Omega,V}$ can be arbitrarily close to $1$.
Theorem \ref{the7} asserts that the efficiency for the first eigenfunction of the Dirichlet Laplacian can be arbitrarily close to $1$.
Among other results, Theorem \ref{the9} provides a quantitative estimate, showing that in case the efficiency for the first eigenfunction of the Dirichlet Laplacian
is close to one, the corresponding first eigenvalue is large. Finally, Theorem \ref{the4} shows that localisation for the torsion function
of the Dirichlet Laplacian implies localisation for the first eigenfunction of this operator.

\begin{definition}\label{def1}
\begin{enumerate}\item[\textup{(i)}] Let $\Omega$ be an open set in $\R^m$
with $0 < |\Omega | < \infty$, and let $f:\Omega\rightarrow [0,\infty)$, with $0<\|f\|_{\infty}<\infty$. The mean to max ratio of $f$  is the real number $\frac{\|f\|_1}{|\Om|\|f\|_{\infty}}$.
\item[\textup{(ii)}] If $v_{\Om,V}$ is the torsion function for $L$, then its efficiency is its mean to max ratio,
$$ \Phi(\Omega,V)=\frac{\|v_{\Om,V}\|_1}{|\Om|\|v_{\Om,V}\|_{\infty}}.$$
\end{enumerate}
\end{definition}
If $V=0$, then $\Phi(\Omega,0)$ is denoted by $\Phi(\Omega),$ which coincides with the definition in \cite{DPGGLB}.

Our first result concerns the comparison of $\Phi(\Omega,V)$ and $\Phi(\Omega)$. In \cite{HLP} it was shown that if $m\ge 2$, then
\begin{equation*}%\label{e17}
\inf\{\Phi(\Omega):\Omega\subset \R^m,\,\Omega\,\,\,\textup{open},\, 0 < |\Omega|<\infty\}=0,
\end{equation*}
and
\begin{equation}\label{e17a}
\sup\{\Phi(\Omega):\Omega\subset \R^m,\,\Omega\,\,\,\textup{open},\, 0 < |\Omega|<\infty\}=1.
\end{equation}
The analogous result for Schr\"odinger operators is stated under (ii) of the theorem below.

\begin{theorem}\label{the1}
\begin{enumerate}
\item[\textup{(i)}]
If $\Omega$ is an open set in $\R^m$ with $0<|\Omega|<\infty$, and if $V: \Omega \to \R^+$ is measurable with
$0\le V \le c,\,c>0$, then
\begin{equation}\label{e7}
\Phi(\Omega,V)\le 2^{2(3m+4)c/\lambda_1(\Omega)}\frac{8c+\lambda_1(\Omega)}{\lambda_1(\Omega)}\bigg(\frac{8c+\lambda_1(\Omega)}{8c}\bigg)^{8c/\lambda_1(\Omega)}\Phi(\Omega),
\end{equation}
and
\begin{equation}\label{e8}
\Phi(\Omega,V)\ge 2^{-2(3m+4)c/\lambda_1(\Omega)}\frac{\lambda_1(\Omega)}{8c+\lambda_1(\Omega)}\bigg(\frac{8c}{8c+\lambda_1(\Omega)}\bigg)^{8c/\lambda_1(\Omega)}\Phi(\Omega).
\end{equation}
Furthermore for fixed $\Omega$ the right-hand sides of \eqref{e7} and \eqref{e8} converge to $\Phi(\Omega)$ as $c\downarrow 0$.
\item[\textup{(ii)}]
Let $(\Omega_n)$ be a sequence of open sets in $\R^m$ with $0 < |\Omega_n|<\infty$, and let $V_n:\Omega_n\to \R^+$ be a sequence of measurable functions. If there exists $\eta<\infty$ such that
$\sup_{n\in\N}\frac{\| V_n \|_{\infty}}{\lambda_1(\Omega_n)} \le \eta,$
then $\lim_{n\rightarrow\infty}\Phi(\Omega_n,V_n)=0$ if and only if $\lim_{n\rightarrow\infty}\Phi(\Omega_n)=0$.
\item[\textup{(iii)}]
If $\Omega$ is a fixed open set in $\R^m$ with $0 < | \Omega | < \infty$, then
\begin{equation}\label{e51}
\sup\{\Phi(\Omega,V): ( V:\Omega\rightarrow \R^+, \textup{measurable})\}=1,
\end{equation}
and
\begin{equation}\label{e52}
\inf\{\Phi(\Omega,V): ( V:\Omega\rightarrow \R^+, \textup{measurable})\}=\Phi(\Omega).
\end{equation}
In fact,
\begin{equation}\label{e50}
1>\Phi(\Omega,c) \ge 1-\frac{2^{(m+4)/2}}{|\Omega|}\int_{\Omega}dx\,e^{-c^{1/2}d_{\Omega}(x)/2}, \qquad \forall c\ge 0,
\end{equation}
where $d_\Omega : \Omega \to \R^+$ is the distance to the boundary function,
\begin{equation*}%\label{e51a}
d_{\Omega}(x)=\min\{|x-y|:y\in \R^m\setminus\Omega\}.
\end{equation*}
\end{enumerate}
\end{theorem}

\medskip

We denote the spectrum of $L$ with Dirichlet boundary conditions by $$\{\lambda_1(\Omega,V)\le \lambda_2(\Omega,V)\le \cdots\},$$
accumulating at infinity only, and choose a corresponding $L^2$-orthonormal basis of eigenfunctions $\{\varphi_{1,\Omega,V}, \varphi_{2,\Omega,V},\cdots\}.$
If the first Dirichlet eigenvalue $\lambda_1(\Omega,V)$ of $L$ has multiplicity $1$, then its corresponding eigenspace is one-dimensional, and $\varphi_{1,\Omega,V}$ is uniquely defined up to a sign.
Since  $\varphi_{1,\Omega,V}$ does not change sign we may choose  $\varphi_{1,\Omega,V}>0.$ In that case we denote the efficiency, or mean to max ratio, of $\varphi_{1, \Omega, V}$ by
\begin{equation*}%\label{e18}
E(\Omega,V)=\frac{\|\varphi_{1, \Omega, V}\|_1}{| \Omega | \, \| \varphi_{1, \Omega, V} \|_{\infty}} .
\end{equation*}
If $V=0$, then $E(\Omega,0)$ is denoted by $E(\Omega),$ which coincides with the definition on p.92 in \cite{sperb}. See also \cite{vdBBDPG}.
We note that if $\Omega$ is connected, then $\lambda_1(\Omega)$ is simple.

By \cite[Theorem 1.2]{vdBFNT} and \cite[Theorem 1]{vdB1} it is possible for $m\ge 2$ to construct, for any $\varepsilon\in(0,1)$, an open connected set $\Omega_{\varepsilon}\subset \R^m$ with $0<|\Om_{\varepsilon}|<\infty$ such that both
$$ \frac{\lambda_1(\Omega_{\varepsilon})\|v_{\Omega_{\varepsilon}}\|_1}{|\Omega_{\varepsilon}|}>1-\varepsilon,$$
and
$$\lambda_1(\Omega_{\varepsilon})\|v_{\Omega_{\varepsilon}}\|_{\infty}<1+\varepsilon.$$
This implies that
$$\Phi(\Omega_{\varepsilon}) \ge \frac{1-\varepsilon}{1+\varepsilon},\qquad\forall\varepsilon\in(0,1),$$
which in turn implies \eqref{e17a}. Given $\varepsilon\in (0,1),$ we were unable to construct a set $\Om_{\varepsilon}$ such that $E(\Omega_{\varepsilon})>1-\varepsilon$. Nevertheless we have the following:
\begin{theorem}\label{the7}
If $m\ge 2$, then
\begin{equation*}%\label{e17c}
\sup\{E(\Omega):\Omega\subset \R^m,\,\Omega\,\,\,\textup{open and connected},\,0 < |\Omega|<\infty\}=1.
\end{equation*}
\end{theorem}

\medskip

By examining the proof of Theorem \ref{the7} in Section \ref{sec4} we see that for any $m\ge 2$, and
$\varepsilon \in (0,1)$ there exists an open, bounded and connected set $\Omega_{\varepsilon}\subset \R^m$
such that (i) $E(\Omega_{\varepsilon})\ge 1-\varepsilon$, and (ii) $\lambda_1(\Omega_{\varepsilon})|\Omega_{\varepsilon}|^{2/m} $
is  large for $\varepsilon$ small. In Theorem \ref{the9} below we show that this is a general phenomenon. That is if $\Omega$ is any open
and connected set in $ \R^m, \,m\ge 2,$ with $0<|\Om|<\infty$ such $E(\Om)$ is close to $1$, then
the eigenfunction is close to its maximum on most of $\Omega$, and
 $\lambda_1(\Omega)|\Omega|^{2/m}$ is large.  We have a similar phenomenon for the torsion function. Throughout we denote by $B(p;R)=\{x\in\R^m:|x-p|<R\}$ the open ball with centre $p$ and radius $R$. We put $B_R=B(0;R)$, and $\omega_m=|B_1|$. For $\Omega$ open with
$0<|\Omega|<\infty$, and $u\in L^1(\Omega),$
$$\fint_{\Omega}u := \frac{1}{|\Omega|} \int_{\Omega} u.$$

\begin{theorem}\label{the9}
Let $m\ge 2$ and let $\Omega$ be a non-empty open set in $\R^m$ with finite Lebesgue measure, $|\Omega|<\infty$.
\begin{itemize}
\item[\textup{(i)}] If $u\in H_0^1(\Omega)\cap L^{\infty}(\Omega)$, $\|u\|_{\infty}>0$, and if
\begin{equation*}%\label{p23}
\fint_{\Omega}u\ge \frac{2\|u\|_{\infty}}{m+2},\,
\end{equation*}
then
\begin{equation}\label{p22}
\bigg(\frac{\omega_m}{|\Omega |}\bigg)^{(m-2)/m}
\bigg(\|u\|_{\infty}-\fint_{\Omega}u\bigg)\int_{\Omega}|\nabla u|^2\ge \frac{4m^2}{(m+2)^2}\omega_m\|u\|^3_{\infty}.
\end{equation}
Equality occurs if and only if $\Omega$ is a ball, and $u$ is a multiple of the torsion function.
\item[\textup{(ii)}]
\begin{equation}\label{p22a}
\|v_{\Omega}\|_{\infty}\le
\frac{(m+2)^2}{4m^2} \bigg(\frac{|\Omega|}{\omega_m}\bigg)^{2/m}(1-\Phi(\Omega)),
\end{equation}
\begin{equation}\label{p22b}
\lambda_1(\Omega)\ge
\frac{4m^2}{(m+2)^2}  \bigg(\frac{\omega_m}{|\Omega |}\bigg)^{2/m}(1-\Phi(\Omega))^{-1}.
\end{equation}
\item[\textup{(iii)}] If $\Omega$ is connected, then
\begin{equation}\label{p22c}
\lambda_1(\Omega)\ge
\frac{4m^2}{(m+2)^2}  \bigg(\frac{\omega_m}{|\Omega |}\bigg)^{2/m}(1-E(\Omega))^{-1}.
\end{equation}
\end{itemize}
\end{theorem}

\medskip

The last result of this paper concerns the localisation of a sequence of torsion functions. We make the following definition.
\begin{definition}\label{def2}
Let $(\Omega_n)$ be a sequence of open sets in $\R^m$ with $0<|\Omega_n|<\infty$, and let
\begin{equation*}%\label{e11}
\mathfrak{A}((\Omega_n))=\bigg\{(A_n): (\forall n\in \N)(A_n\subset\Omega_n, A_n\, \textup{measurable}), \lim_{n\rightarrow\infty}\frac{|A_n|}{|\Omega_n|}=0\bigg\}.
\end{equation*}
Let $1\le p<\infty$. For $n\in \N,\,f_n\in L^p(\Omega_n),f_n\ge 0,$ and $f_n\ne 0$, we define
\begin{equation}\label{e12}
\kappa=\sup\bigg\{\limsup_{n\rightarrow\infty}\frac{\|f_n{\bf 1}_{A_n}\|_p^p}{\|f_n\|_p^p}:(A_n)\in\mathfrak{A}((\Omega_n))\bigg\}.
\end{equation}
\begin{enumerate}
\item[\textup{(i)}]We say that the sequence $\big(f_{n}\big)$ $\kappa$-localises in $L^p$ if $0<\kappa<1$.
\item[\textup{(ii)}]We say that the sequence $\big(f_{n}\big)$ localises in $L^p$ if $\kappa=1$.
\item[\textup{(iii)}]We say that the sequence $\big(f_{n}\big)$ does not localise in $L^p$ if $\kappa=0$.
\end{enumerate}
\end{definition}

Using Cantor's diagonalisation we see the supremum in \eqref{e12} is in fact a maximum. That is we have a maximising sequence in $\mathfrak{A}((\Omega_n))$.
To show that $\big(f_n\big)$ localises in $L^p$ is equivalent to showing the existence of a sequence of measurable sets $A_n\subset\Omega_n,\,n\in \N$ such that
\begin{equation}\label{ea12}
\lim_{n\rightarrow\infty}\frac{|A_n|}{|\Omega_n|}=0,\,\,\, \lim_{n\rightarrow\infty}\frac{\|f_n{\bf 1}_{A_n}\|^p_p}{\|f_n\|^p_p}=1.
\end{equation}
To show that $(f_n)$ does not localise in $L^p$ is equivalent to showing that for any sequence $(A_n)$ of measurable
sets $A_n\subset\Omega_n,\,n\in \N$ the following implication holds:
\begin{equation}\label{ea13}
\lim_{n\rightarrow\infty}\frac{|A_n|}{|\Omega_n|}=0 \Rightarrow \lim_{n\rightarrow\infty}\frac{\|f_n{\bf 1}_{A_n}\|_p}{\|f_n\|_p}=0.
\end{equation}

If $\big(f_{n}\big)$ $\kappa$-localises in $L^p$ then there is a sequence  $(A_n)\in\mathfrak{A}((\Om_n))$ which (asymptotically) supports a fraction $\kappa$ of $\|f_n\|_p^p$.
Given such a maximising sequence $(A_n)$ it is possible to construct a sequence $(\tilde{A_n})\in\mathfrak{A}((\Om_n))$ which (asymptotically) supports a fraction $\tilde{\kappa}$ of $\|f_n\|_p^p$
with $0<\tilde{\kappa}<\kappa$. Hence the requirement of the supremum in the definition of $\kappa$ in \eqref{e12}. In Section \ref{sec2} we analyse two examples in detail where localisation in $L^1$ and $\kappa$-localisation in $L^1$ occur for a family of torsion functions for Schr\"odinger operators (Example \ref{exa1}), and for a family of torsion functions for Dirichlet Laplacians (Example \ref{exa2}).

\medskip

The theorem below asserts that localisation or  $\kappa$-localisation for the torsion function in $L^1$ implies localisation for the corresponding first Dirichlet eigenfunction in $L^2$.
\begin{theorem}\label{the4}
Let $(\Omega_n)$ be a sequence of open sets in $\R^m$ with $0 < |\Omega_n|<\infty$ and let $V_n:\Omega_n\to \R^+$ be a sequence of measurable functions.
If $(v_{\Omega_n,V_n})$ either localises or $\kappa$-localises in $L^1$, and if $\lambda_1(\Omega_n,V_n)$ has multiplicity $1$, then $(\varphi_{1,\Omega_n,V_n})$ localises in $L^2$, and $\lim_{n\rightarrow\infty}\Phi(\Om_n,V_n)=\lim_{n\rightarrow\infty}E(\Om_n,V_n)=0$.
\end{theorem}

It follows from Theorem 1.1 in \cite{DPGGLB} that, when restricting $\Omega$ to be open, bounded and convex in $\R^m,\, m\ge 1$, one has
\begin{equation}\label{e70}
\frac{2}{m(m+2)}\le \inf\{\Phi(\Omega): \Omega\subset \R^m,\,\Omega\,\,\,\textup{open,\,convex},\,0<|\Omega|<\infty\}.
\end{equation}
We see by \eqref{e70} that for any sequence $(\Omega_n)$ of elongating open, bounded and convex sets in $\R^m$,
$(v_{\Omega_n})$ has non-vanishing efficiency, and so by Theorem \ref{the4}, $(v_{\Om_n})$ is not localising or $\kappa$-localising in $L^1$.
This is in contrast with the results of \cite{vdBBDPG}, where localisation of a sequence of first Dirichlet eigenfunctions was obtained for a wide class of elongating, open, bounded and convex sets in $\R^m$.
Further examples demonstrating different behaviour of the torsion function, and the first eigenfunction of
the Dirichlet Laplacian around their respective maxima for elongated convex planar domains have been constructed in \cite{Beck}.

\medskip

\medskip

This paper is organised as follows. Examples \ref{exa1} and \ref{exa2} will be analysed in Section \ref{sec2}. The proofs of Theorems \ref{the1}, \ref{the7}, \ref{the9} and \ref{the4} will be given in Section \ref{sec3}--\ref{sec6} respectively.

\section{ Examples \label{sec2}}

In Example \ref{exa1} below we analyse localisation and $\kappa$-localisation in $L^1$ for a family of Schr\"odinger operators in one dimension,
parametrised by three real numbers, $\nu > 1$, $0 < \alpha < 1$, and $c >0$.
\begin{example}\label{exa1}
Let $\Omega=(-1, 1)$, $\nu > 1$, and $0 < \varepsilon < 1$.
Denote by $V(x) \equiv V_{\nu, \varepsilon}(x)$ the potential
$$
V(x) = \nu^2 {\bf 1}_{(-1, -\varepsilon)}(x) +
\nu^2 {\bf 1}_{(\varepsilon, 1)}(x)
\, , \qquad x \in (-1, 1)\, ,
$$
and by $ v_{\nu, \varepsilon}(x) \equiv v_{(-1,1), V_{\nu, \varepsilon}}(x)$ the
torsion function for $- \Delta + V$ on $(-1, 1)$ with Dirichlet boundary conditions.
If $\varepsilon_\alpha(\nu) = \frac{c}{\nu^\alpha}$ with $0 < \alpha < 1$,  $c>0$, and
$\nu > 1$ sufficiently large so that $0 < \varepsilon_\alpha(\nu) < 1$, then the following holds:
\begin{itemize}
\item[\textup{(i)}]If $\frac 23 < \alpha < 1$, then $(v_{n, \varepsilon_\alpha(n)})$ does not localise in $L^1$, $\kappa=0.$
\item[\textup{(ii)}]If $0 <\alpha < \frac 23$, then $(v_{n, \varepsilon_\alpha(n)})$ localises in $L^1$, $\kappa=1.$
\item[\textup{(iii)}]If $\alpha = \frac 23$, then $(v_{n, \varepsilon_\alpha(n)})$ $\kappa_c$-localises in $L^1$ with
\begin{equation*}%\label{exa 2}
\kappa_c=\frac{c^3/3}{1+c^3/3}.
\end{equation*}
\end{itemize}
\end{example}
\begin{proof}
Since the potential $V$ is even, so is the torsion
function $v \equiv v_{\nu, \varepsilon}$ and hence it suffices to determine
$v$ on $[0, 1]$. On the interval $[ 0, \varepsilon]$,
the function $v$ is of the form
\begin{equation*}%\label{v nu,1}
v_{1}(x) := - \frac{1}{2}x^2 + \gamma\, , \quad 0 \le x \le \varepsilon\, , \qquad \quad \gamma \equiv \gamma_{\nu, \varepsilon},
\end{equation*}
whereas on the interval $[\varepsilon, 1]$, we make the following Ansatz:
\begin{equation}\label{ansatz v_2}
v_{2}(x) := \frac{1}{\nu^2}
- \alpha e^{\nu x}
+ \beta e^{-\nu x}\, ,
\quad \varepsilon  \le x \le 1\, ,
 \qquad \quad  \alpha \equiv \alpha_{\nu, \varepsilon}, \   \beta \equiv \beta_{\nu, \varepsilon}.
\end{equation}
It is straightforward to verify that
$$
(-\Delta + V) v_{ 1} = - \Delta v_1 =1, \quad 0 \le x \le  \varepsilon ,
$$
and
$$
(-\Delta + V) v_{2} = - \nu^2(v_2 - \frac{1}{\nu^2})  + \nu^2 v_2 =1, \quad \varepsilon \le x \le 1.
$$
The constants $\gamma,$ $\alpha$, and $\beta$ are determined by the boundary condition
$v_{ 2}(1) = 0$, and the matching conditions
$$
v_{1}(\varepsilon) = v_{ 2}(\varepsilon)\, ,\qquad
v_{1}'(\varepsilon) = v_{ 2}'(\varepsilon)\, ,
$$
where $'$ denotes the derivative with respect to the variable $x$.
The boundary condition $v_{2}(1) = 0$ yields
$\frac{1}{\nu^2} - \alpha e^{\nu} + \beta e^{-\nu} =0$ or
\begin{equation}\label{formula beta 1}
\beta = \alpha e^{2\nu} -
\frac{1}{\nu^2} e^{\nu}\, .
\end{equation}
The matching condition
$v_{1}'(\varepsilon) =
v_{2}'(\varepsilon)$ reads
$\varepsilon = \nu \alpha e^{\nu \varepsilon} + \nu \beta e^{-\nu \varepsilon}$ or,  with \eqref{formula beta 1},
\begin{equation}\label{formula alpha nu 1}
\alpha = \frac{\varepsilon +\frac{1}{\nu} e^{\nu ( 1-\varepsilon)}}{\nu(e^{\nu \varepsilon} +  e^{2\nu - \nu \varepsilon})}
 = \frac{1}{\nu^2 e^{\nu}}
\frac{1 + \nu \varepsilon e^{-\nu(1 - \varepsilon)}}{1 + e^{-2\nu(1 - \varepsilon)}}\,.
\end{equation}
The leading term of $\alpha$ is thus $\frac{1}{\nu^2 e^{\nu}}$,
\begin{equation}\label{formula alpha nu 2}
\alpha = \frac{1}{\nu^2 e^{\nu}} + \frac{ \varepsilon}{\nu e^{\nu(2 -  \varepsilon)}} \frac{1 - \frac{1}{\nu \varepsilon } e^{-\nu(1 - \varepsilon)}}{1 + e^{-2\nu(1 - \varepsilon)}} .
\end{equation}
Combining \eqref{formula beta 1} and \eqref{formula alpha nu 2} we obtain the following formula for $\beta$,
\begin{equation}\label{formula beta 2}
\beta =
\frac{\varepsilon e^{\nu \varepsilon}}{\nu} \frac{1 - \frac{1}{\nu \varepsilon} e^{-\nu(1- \varepsilon)}}{1 + e^{-2\nu(1 - \varepsilon)}}
= \frac{\varepsilon e^{\nu \varepsilon}}{\nu}
- \frac{ e^{\nu \varepsilon}}{\nu^2} \frac{1}{e^{\nu(1- \varepsilon)}}
\frac{1 + \nu \varepsilon e^{-\nu(1 - \varepsilon)}}{1 + e^{-2\nu(1 - \varepsilon)}} .
\end{equation}
Finally the matching condition $v_{1}(\varepsilon) = v_{2}(\varepsilon)$ reads as
$$
- \frac{1}{2} \varepsilon^2 + \gamma =
\frac{1}{\nu^2}  - \alpha e^{\nu \varepsilon}
+ \beta e^{-\nu \varepsilon}\, ,
$$
which, when combined with \eqref{formula alpha nu 1} and \eqref{formula beta 2},
yields the following formula for $\gamma$
\begin{equation}\label{formula gamma}
\gamma = \frac{1}{2} \varepsilon^2 + \frac{\varepsilon}{\nu} + \frac{1}{\nu^2}
-  \frac{2}{\nu^2 e^{\nu(1- \varepsilon)}} \frac{1 + \nu \varepsilon e^{-\nu(1 - \varepsilon)}}{1 + e^{-2\nu(1 - \varepsilon)}} .
\end{equation}
Now let us compute
$$
\frac{\int_{\varepsilon }^1 v}{\int_0^1 v}
= \frac{\int_{\varepsilon }^1 v_2}{\int_{0}^{\varepsilon } v_1 + {\int_{\varepsilon }^1 v_2}}.
$$
We have $\int_{0}^{\varepsilon } v_1 = (- \frac 16 x^3 + \gamma x )|_0^\varepsilon$, yielding
\begin{equation}\label{formula integral 1}
\int_{0}^{\varepsilon} v_1 =
\frac{\varepsilon^3}{3} + \frac{\varepsilon^2}{\nu} + \frac{\varepsilon}{\nu^2}
-  \frac{2 \varepsilon}{\nu^2 e^{\nu(1- \varepsilon)}} \frac{1 + \nu \varepsilon e^{-\nu(1 - \varepsilon)}}{1 + e^{-2\nu(1 - \varepsilon)}} .
\end{equation}
Similarly, we compute
$$
\int_{\varepsilon }^1 v_2 = \frac{1 - \varepsilon}{\nu^2} - \frac{\alpha}{\nu}( e^{\nu} - e^{\nu \varepsilon})
- \frac{\beta}{\nu}( e^{-\nu} - e^{-\nu \varepsilon}) ,
$$
which yields
\begin{equation}\label{formula integral 2}
\int_{\varepsilon }^1 v_2 = \frac{1}{\nu^2} - \frac{1}{\nu^3}
-  \frac{2 \varepsilon}{\nu^2 e^{\nu(1- \varepsilon)}} \frac{1 -\frac{1}{ \nu \varepsilon} e^{-\nu(1 - \varepsilon)}}{1 + e^{-2\nu(1 - \varepsilon)}} .
\end{equation}
Combining \eqref{formula integral 1} and \eqref{formula integral 2}, one obtains in the case where
$\varepsilon = \frac{c}{\nu^\alpha}$,
\begin{equation}\label{integral v}
 \int_{0}^{\varepsilon} v_1   + \int_{\varepsilon }^1 v_2
 = \frac{c^3}{3\nu^{3 \alpha}} + \frac{1}{\nu^2} + O\big(\nu^{-1 - 2 \alpha}\big),\,\nu \rightarrow\infty,
\end{equation}
and hence
$$
\lim_{\nu \to \infty} \frac{\int_{\varepsilon }^1 v_2}{ \int_{0}^{\varepsilon} v_1   + \int_{\varepsilon }^1 v_2 }
= \begin{cases} 1 \qquad \qquad \ \text{if} \ \ \frac 23 < \alpha < 1,\\
 0 \qquad \qquad \ \text{if} \ \ 0 < \alpha < \frac 23, \\
 \frac{1}{1 + c^3/3} \qquad \text{if} \ \ \alpha = \frac 23 .
\end{cases}
$$
This implies that
\begin{equation}\label{limit of quotient}
\lim_{\nu \to \infty} \frac{\int_{-\varepsilon }^{\varepsilon } v}{ \int_{-1}^1 v }
= 1 - \lim_{\nu \to \infty} \frac{\int_{\varepsilon }^1 v_2}{ \int_{0}^{\varepsilon} v_1   + \int_{\varepsilon }^1 v_2 }
= \begin{cases} 0 \qquad \qquad \ \text{if} \ \  \frac 23 < \alpha < 1,\\
 1 \qquad \qquad \ \text{if} \ \ 0 < \alpha < \frac 23, \\
 \frac{c^3/3}{1 + c^3/3} \qquad \text{if} \ \ \alpha = \frac 23 .
\end{cases}
\end{equation}
To prove (i) we let $A_n \subset (-1, 1),\,n\in\N$ be an arbitrary sequence of measurable sets which satisfy $\lim_{n\rightarrow\infty} |A_n|=0.$
Since $v_n \equiv v_{n, \varepsilon_\alpha(n)}$ is even, it follows that, with $\varepsilon_n \equiv \varepsilon_\alpha(n)$,
$$
\int_{A_n} v_n \le 2 \int_0^{\varepsilon_n} v_n + 2  \int_{\varepsilon_n}^{2\varepsilon_n} v_n
+   \int_{A_n \cap [2\varepsilon_n, 1]} v_n
+   \int_{A_n \cap [-1, -2\varepsilon_n]} v_n.
$$
First note that  by \eqref{integral v}
\begin{equation}\label{integral v 1}
 \int_{0}^{\varepsilon_n} v_n   + \int_{\varepsilon_n }^1 v_n = \frac{1}{n^2} + O\big(n^{-3\alpha}\big),\,\nu \rightarrow\infty.
\end{equation}
By \eqref{limit of quotient}
$\lim_{n \to \infty} \frac{\int_0^{\varepsilon_n } v_n}{ \int_{0}^1 v_n } = 0$. Since for $n$ sufficiently large,
$v_n(x)$ is decreasing on $[0, 1]$ and by \eqref{formula gamma},
$v_n(\varepsilon_n) =  \frac{ \varepsilon_n}{n} + O(\frac{1}{n^2})$, it follows from \eqref{integral v 1} that
$$
\lim_{n \to \infty} \frac{\int_{\varepsilon_n }^{2\varepsilon_n } v_n}{ \int_{0}^1 v_n } \le
\lim_{n \to \infty} \frac{ \varepsilon_n v_n(\varepsilon_n)}{ \int_{0}^1 v_n } = 0 .
$$
Using once more that for $n$ sufficiently large, $v_n(x)$ is decreasing on $[0, 1]$, one has
$v_n(2\varepsilon_n) \le  \frac{1}{n^2} + O( \frac{1}{n^{1+\alpha}} e^{-n(1-\alpha)})$ (see \eqref{ansatz v_2}, \eqref{formula beta 2}), and
then infers from \eqref{integral v 1} that
$$
\lim_{n \to \infty} \frac{\int_{A_n \cap [2\varepsilon_n, 1]} v_n}{ \int_{0}^1 v_n}
\le \lim_{n \to \infty} \frac{v_n(2\varepsilon_n) |A_n|}{ \int_{0}^1 v_n}  = 0 .
$$
Similarly, one has
$$
\lim_{n \to \infty} \frac{\int_{A_n \cap [-1, - 2\varepsilon_n]} v_n}{ \int_{0}^1 v_n}  = 0 .
$$
To prove (ii) we obtain a lower bound for the supremum in \eqref{e12} by choosing the sequence $A_n = (- \varepsilon_n, \varepsilon_n),\,n\in \N$.
By \eqref{limit of quotient} one has $\kappa = 1$ in this case.

To prove (iii) we obtain a lower bound for the supremum in \eqref{e12} by choosing the sequence $A_n = (- \varepsilon_n, \varepsilon_n),\,n\in \N$.
Hence by  \eqref{limit of quotient} we then get
\begin{equation}\label{lower bound kappa}
\kappa \ge  \kappa_c = \frac{c^3/3}{1 + c^3/3} .
\end{equation}
To prove the reverse inequality we let $A_n\subset (-1, 1),\,n\in\N$ be an arbitrary sequence of measurable sets which satisfy $\lim_{n\rightarrow\infty} |A_n|=0.$
In view of \eqref{lower bound kappa} we may assume without loss of generality that
$A_n$ is symmetric, $A_n = - A_n$, and that  $[-\varepsilon_n , \varepsilon_n ] \subset A_n $ for any $n$ (sufficiently large).
It then suffices to show that
$$
\lim_{n \to \infty} \frac{\int_{A_n \cap [\varepsilon_n, 1]} v_n}{ \int_{0}^1 v_n}  = 0.
$$
By \eqref{integral v}
$$
 \int_{0}^{\varepsilon_n} v_n   + \int_{\varepsilon_n }^1 v_n
 = \frac{1 + c^3/3}{n^2} + O(n^{-7/3}).
 $$
 As in the proof of item (i), we estimate
 $$
\int_{A_n \cap [\varepsilon_n, 1]} v_n  \le    \int_{\varepsilon_n}^{2\varepsilon_n} v_n +   \int_{A_n \cap [2\varepsilon_n, 1]} v_n,
$$
 and obtain
 $$
\lim_{n \to \infty} \frac{\int_{\varepsilon_n }^{2\varepsilon_n } v_n}{ \int_{0}^1 v_n } \le
\lim_{n \to \infty} \frac{ \varepsilon_n v_n(\varepsilon_n)}{ \int_{0}^1 v_n } = 0 .
 $$
 and
 $$
\lim_{n \to \infty} \frac{\int_{A_n \cap [2\varepsilon_n, 1]} v_n}{ \int_{0}^1 v_n}
\le \lim_{n \to \infty} \frac{v_n(2\varepsilon_n) |A_n|}{ \int_{0}^1 v_n}  = 0 .
 $$
 Altogether we proved that $\kappa = \kappa_c$ and that
the supremum $\kappa$ is attained by the sequence $(A_n)$ with $(-\varepsilon_n, \varepsilon_n)$.
\end{proof}

In Example \ref{exa2} below we analyse localisation and $\kappa$-localisation in $L^1$ for a family of sequences of open sets in $\R^m$, parametrised by three real positive $\alpha,\beta$ and $c$.

\begin{example}\label{exa2}
Let $m\ge 1 $, and let $\Omega_n,\,n\in \N,$ be the union of $n+1$ open balls $B(p_1;n^{-\alpha}),...,B(p_n;n^{-\alpha}), B(p_{n+1};cn^{-\beta})$ with centers $p_1,...,p_{n+1}$ respectively. Let $c>0$, and let
\begin{equation}\label{eb0}
|p_i-p_j|\ge 2+c, \quad \forall(i,j)\in \{\{1,...,n\}^2, i\ne j\},
\end{equation}
where
\begin{equation}\label{eb1}
\beta>\alpha-\frac1m\ge 0.
\end{equation}
\begin{itemize}
\item[\textup{(i)}]If $\beta>\alpha-\frac{1}{m+2}$, then $(v_{\Omega_{n}})$ does not localise in $L^1$, $\kappa=0.$
\item[\textup{(ii)}]If $\beta<\alpha-\frac{1}{m+2}$, then $(v_{\Omega_{n}})$ localises in $L^1$, $\kappa=1.$
\item[\textup{(iii)}]If $\beta=\alpha-\frac{1}{m+2}$, then $(v_{\Omega_{n}})$ $\kappa_c$-localises in $L^1$ with
\begin{equation}\label{e13}
\kappa_c=\frac{c^{m+2}}{1+c^{m+2}}.
\end{equation}
\end{itemize}
\end{example}

An example of a sequence of open, bounded, simply connected planar sets with fixed measure $1$ for which the torsion function is $\kappa$-localising in $L^1$ for some $0<\kappa<1$ has been given in Theorem 2 of \cite{vdBK}.

\begin{proof}
First observe that condition \eqref{eb0}  guarantees that the $n+1$ open balls do not intersect pairwise. The first inequality in \eqref{eb1} guarantees that the measure of $B(p_{n+1};cn^{-\beta})$ is negligible compared with the measure of $\Om_n$ in the limit $n\rightarrow\infty$. The second inequality in \eqref{eb1} implies that $|\Om_n|$ remains bounded for large $n$.
The torsion function for the open ball $B(p;R)$ is given by
\begin{equation*}%\label{eb2}
v_{B(p;R)}(x)=\frac{R^2-|x-p|^2}{2m}.
\end{equation*}
Hence
\begin{equation*}%\label{eb3}
\|v_{B(p;R)}\|_1=\rho_m R^{m+2},
\end{equation*}
where $\rho_m=\frac{\omega_m}{m(m+2)}$ is the torsional rigidity for a ball in $\R^m$ with radius $1$. Furthermore
\begin{equation*}%\label{eb4}
\|v_{B(p;R)}\|_{\infty}=\frac{R^2}{2m},
\end{equation*}
and
\begin{equation}\label{eb6}
\|v_{\Om_n}\|_1=\rho_m\big(c^{m+2}n^{-(m+2)\beta}+n^{1-(m+2)\alpha}\big).
\end{equation}
To prove (i) we let $A_n\subset\Om_n,\,n\in\N$ be an arbitrary sequence of measurable sets which satisfy $\lim_{n\rightarrow\infty}\frac{|A_n|}{|\Om_n|}=0.$  We have
\begin{align}\label{eb7}
\int_{A_n}v_{\Om_n}&=\int_{A_n\cap B(p_{n+1};cn^{-\beta})}v_{\Om_n}+\int_{A_n\cap\big(\cup_{i=1}^n B(p_i;n^{-\alpha})\big)}v_{\Om_n}\nonumber \\ &
\le \int_{ B(p_{n+1};cn^{-\beta})}v_{\Om_n}+\int_{A_n}\|v_{B(p_1;n^{-\alpha})}\|_{\infty}\nonumber \\ &
=\rho_mc^{m+2}n^{-(m+2)\beta}+\frac{1}{2m}|A_n|n^{-2\alpha}.
\end{align}
By \eqref{eb6} we have
\begin{align*}%\label{eb8}
\frac{\int_{A_n}v_{\Om_n}}{\|v_{\Om_n}\|_1}&\le c^{m+2}n^{(m+2)(\alpha-\beta)-1}+\frac{\omega_m}{2m\rho_m}\frac{|A_n|}{|\Om_n|}\frac{n^{-2\alpha}\big(c^mn^{-m\beta}+n^{1-m\alpha}\big)}{n^{1-(m+2)\alpha}}\nonumber \\ &
\le c^{m+2}n^{(m+2)(\alpha-\beta)-1}+\frac{\omega_m}{2m\rho_m}\big(c^m+1\big)\frac{|A_n|}{|\Om_n|},
\end{align*}
where we have used \eqref{eb1} to bound the second term in the previous line.
By the hypothesis for $\beta$ under (i) and the hypothesis on $(A_n)$ above we conclude
\begin{equation*}%\label{eb9}
\lim_{n\rightarrow\infty}\frac{\int_{A_n}v_{\Om_n}}{\|v_{\Om_n}\|_1}=0.
\end{equation*}
This proves the implication under \eqref{ea13}, and concludes the proof of (i).

\medskip

\noindent To prove (ii) let $A_n=B(p_{n+1};cn^{-\beta}),\,n\in \N$. This gives,
\begin{align}\label{eb10}
\kappa &\ge \limsup_{n\rightarrow \infty}\frac{\int_{B(p_{n+1};cn^{-\beta})}v_{\Om_n}}{\|v_{\Om_n}\|_1}\nonumber \\ &
=\limsup_{n\rightarrow \infty}\frac{c^{m+2}n^{-(m+2)\beta}}{c^{m+2}n^{-(m+2)\beta}+n^{1-(m+2)\alpha}}.
\end{align}
By the hypothesis for $\beta$ under (ii) we have
\begin{equation*}%\label{eb11}
\kappa\ge \limsup_{n\rightarrow \infty}\frac{c^{m+2}n^{-(m+2)\beta}}{c^{m+2}n^{-(m+2)\beta}+n^{1-(m+2)\alpha}}=1.
\end{equation*}
Both requirements under \eqref{ea12} are satisfied. This concludes the proof of (ii).

\medskip

\noindent To prove (iii) we obtain a lower bound for $\kappa_c$ by choosing the sequence $A_n=B(p_{n+1};cn^{-\beta}),\,n\in \N$. By \eqref{eb10} and the hypothesis for $\beta$ under (iii) we get
\begin{equation}\label{eb12}
\kappa\ge\frac{c^{m+2}}{1+c^{m+2}}.
\end{equation}
To prove the reverse inequality we let $A_n\subset\Om_n,\,n\in\N$ be an arbitrary sequence of measurable sets which satisfy $\lim_{n\rightarrow\infty}\frac{|A_n|}{|\Om_n|}=0.$
By \eqref{eb6}, \eqref{eb7}, and the hypothesis for $\beta$ under (iii),
\begin{align}\label{eb13}
\frac{\int_{A_n}v_{\Om_n}}{\|v_{\Om_n}\|_1}&\le \frac{c^{m+2}}{1+c^{m+2}}+\frac{\omega_m}{2m\rho_m}\frac{|A_n|}{|\Om_n|}\big(c^mn^{-2/(m+2)}+1\big).
\end{align}
By taking the $\limsup_{n\rightarrow\infty}$ in both sides of the inequality in \eqref{eb13},
\begin{equation}\label{eb14}
\limsup_{n\rightarrow\infty}\frac{\int_{A_n}v_{\Om_n}}{\|v_{\Om_n}\|_1}\le \frac{c^{m+2}}{1+c^{m+2}}.
\end{equation}
By taking the supremum over all sequences $(A_n)\in \mathfrak{A}((\Om_n))$ we find by \eqref{eb14},
\begin{equation}\label{eb15}
\kappa\le \frac{c^{m+2}}{1+c^{m+2}}.
\end{equation}
This proves by \eqref{eb12} and \eqref{eb15} that $(v_{\Om_n})$ is $\kappa_c$-localising with $\kappa_c$ given by \eqref{e13}. This concludes the proof of (iii).
\end{proof}

\section{Proof of Theorem \ref{the1} \label{sec3}}

To prove Theorem \ref{the1} we first need to establish some auxiliary results.
It is well known that the torsion function
can be expressed in terms of the heat kernel of $L$.
Let $\Omega \subset \R^m$ be open with
$0 < |\Omega| < \infty$, and let $V:\Omega\rightarrow \R^+$ be measurable. Denote by
$p_{\Omega,V}(x,y;t)$, $x\in \Omega, y\in \Omega,t > 0$ the
heat kernel of
\begin{equation*}%\label{e23}
\frac{\partial u}{\partial t} = Lu,  \qquad u\in H_0^1(\Omega).
\end{equation*}
The torsion function of $L$ then satisfies
\begin{equation}\label{e24}
v_{\Omega, V}(x) = \int_{\Omega} dy \int_{\R^+} dt\, p_{\Omega,V}(x,y;t),\qquad \forall x \in \Omega\ .
\end{equation}
In case $V=0$ we write $p_{\Omega}$ for $p_{\Omega,V}$.

Recall the Feynman-Kac formula (\cite{BS}) for non-negative, measurable potentials $V:\Omega\rightarrow \R^+,$
\begin{align*}%\label{e25}
p_{\Omega,V}(x,y;t) = &p_{\R^m}(x,y;t)\nonumber \\ & \times \E[e^{- \int_0^t V(\beta(s))ds} \big(\Pi_{s\in[0,t]}{\bf1}_{\Omega}(\beta(s))\big) \, : \,  \beta(0)=x, \beta(t) = y ],
\end{align*}
where $\beta$ is a Brownian bridge.
Hence if $V_1: \Omega \to \R^+$ and $V_2: \Omega \to \R^+$ are measurable functions with $0 \le V_2 \le  V_1$, then

\begin{equation}\label{e26}
0\le p_{\Omega,V_1}(x,y;t)\le p_{\Omega,V_2}(x,y;t),\qquad \forall x\in \Omega, \forall y\in \Omega, \forall t>0.
\end{equation}
Since $V\ge 0$ we then conclude that
\begin{equation*}%\label{e27}
0 \le v_{\Omega, V}(x) \le v_{\Omega}(x)\le \mathfrak{c}_m\lambda_1(\Omega)^{-1}, \qquad  \forall x \in \Omega,
\end{equation*}
with $\mathfrak{c}_m$ given by \eqref{e5}.

\begin{lemma}\label{lem1} If $0 < |\Omega|<\infty$, and
if $V: \Omega \to \R^+$ is measurable, then
\begin{equation}\label{e29}
\lambda_{1}(\Omega, V)^{-1} \le
\|v_{\Omega,V}\|_{\infty}
\le (4+3m\log 2)\lambda_{1}(\Omega, V)^{-1}.
\end{equation}
\end{lemma}
\medskip
Lemma \ref{lem1} implies that
\begin{align}\label{e29a}
\mathfrak{d}_m:&=\sup\{\lambda_1(\Omega,V)\|v_{\Omega,V}\|_{\infty}\nonumber \\ &
\,\,\,\,:\Omega\,\textup{ open in }\,\R^m,\,0<|\Omega|<\infty,\, V:\Omega\rightarrow \R^+, \textup{measurable} \}<\infty.
\end{align}
By choosing $V=0$ in the expression under the supremum in the right-hand side of \eqref{e29a} we see that $\mathfrak{c}_m\le \mathfrak{d}_m$.

\medskip

\noindent{\it Proof of Lemma \textup{\ref{lem1}.}}
To prove the upper bound in \eqref{e29} we note that Lemma 1 and its proof in \cite{vdBC} hold with $\lambda=\lambda_1(\Omega,V)$, and Lemma 2 and its proof in \cite{vdBC} hold for the semigroup associated with $L$. Finally Lemma 3 and Theorem 1 and their proofs in \cite{vdBC} hold with
$\lambda=\lambda_1(\Omega,V)$. This proves the upper bound in \eqref{e29}.
It remains to prove the lower bound.
Let $\Omega_R=\Omega\cap B(0;R)$. Then
$0<|\Omega_R|\le \omega_mR^m$.
Let $L_R$ be the restriction of $-\Delta+V$ acting in $L^2(\Omega_R)$ with Dirichlet boundary conditions on $\partial \Omega_R$.
If we denote $V_R={\bf1}_{B(p;R)}V,$ then $L_R$ is also the operator $-\Delta+V_R$ acting in $L^2(\Omega_R)$ with Dirichlet boundary conditions on $\partial \Omega_R$.
Then $L_R$ is self-adjoint, and its spectrum is discrete. Since the first Dirichlet eigenfunction $\varphi_{1,\Omega_R,V_R}$ is non-negative, and
$\| \varphi_{1,\Omega_R,V_R} \|_2 =1 $,
one has, by the Cauchy-Schwarz inequality,
\begin{equation}\label{e31}
0<\int_{\Omega_R}\varphi_{1,\Omega_R,V_R}\le |\Omega_R|^{1/2}.
\end{equation}
By self-adjointness
\begin{align}\label{e32}
\int_{\Omega_R}\varphi_{1,\Omega_R,V_R}
&=\int_{\Omega_R}\varphi_{1,\Omega_R,V_R}L_Rv_{\Omega_R,V_R}=
\int_{\Omega_R}\big(L_R\varphi_{1,\Omega_R,V_R}\big)v_{\Omega_R,V_R}\nonumber \\
&=\lambda_1(\Omega_R,V_R)\int_{\Omega_R}\varphi_{1,\Omega_R,V_R}v_{\Omega_R,V_R}\nonumber \\
& \le \lambda_1(\Omega_{R},V_R)\|v_{\Omega_R,V_R}\|_{\infty}\int_{\Omega_R}\varphi_{1,\Omega_R,V_R}.
\end{align}
By \eqref{e31} -- \eqref{e32} we conclude that
\begin{equation*}%\label{e33}
\lambda_1(\Omega_R,V_R) \, \|v_{\Omega_R,V_R}\|_{\infty} \ge 1.
\end{equation*}
Since $\|v_{\Omega_R,V_R}\|_{\infty}\le \|v_{\Omega,V}\|_{\infty}$ we have
\begin{equation*}%\label{e34}
\lambda_1(\Omega_R,V_R) \|v_{\Omega,V}\|_{\infty}\ge 1.
\end{equation*}
The assertion follows since $R\mapsto \lambda_1(\Omega_R,V_R)$ is decreasing to $\lambda_1(\Omega,V)$ as $R\rightarrow\infty$.
\hfill$\square$

\begin{lemma}\label{lem2}
Let $\Omega\subset \R^m$ be open with
$0 <  |\Omega | < \infty$. If $V : \Omega \rightarrow \R^+$ is measurable with
$0 \le c_1 \le V(x) \le c_2 < \infty$, $x \in \Omega$,
then
\begin{align}\label{e35}
 v_{\Omega,c_2}(x) \le
 v_{\Omega,V}(x)\le v_{\Omega,c_1}(x)\ ,
\end{align}
and for any $0 <c < \infty$, $x\in \Omega$,
\begin{align}\label{e36}
 v_{\Omega,c}(x)\ge 2^{-2(3m+4)c/\lambda_1(\Omega)}\frac{\lambda_1(\Omega)}{8c+\lambda_1(\Omega)}\bigg(\frac{8c}{8c+\lambda_1(\Omega)}\bigg)^{8c/\lambda_1(\Omega)}v_{\Omega}(x).
\end{align}
Hence
\begin{equation}\label{e37}
 \|v_{\Omega,c}\|_1\ge 2^{-2(3m+4)c/\lambda_1(\Omega)}\frac{\lambda_1(\Omega)}{8c+\lambda_1(\Omega)}\bigg(\frac{8c}{8c+\lambda_1(\Omega)}\bigg)^{8c/\lambda_1(\Omega)}\|v_{\Omega}\|_1,
\end{equation}
and
\begin{equation}\label{e38}
 \|v_{\Omega,c}\|_{\infty}\ge 2^{-2(3m+4)c/\lambda_1(\Omega)}\frac{\lambda_1(\Omega)}{8c+\lambda_1(\Omega)}\bigg(\frac{8c}{8c+\lambda_1(\Omega)}\bigg)^{8c/\lambda_1(\Omega)}\|v_{\Omega}\|_{\infty}.
\end{equation}
Furthermore the right-hand side of \eqref{e36} converges to $v_{\Omega}(x)$ as $c\downarrow 0$.

\end{lemma}

\begin{proof}
The two inequalities in \eqref{e35} follow immediately from \eqref{e26}, and the hypothesis $0\le c_1 \le V \le c_2$. To prove inequality \eqref{e36}, note that for any $T>0$, and any $c\ge 0$,
\begin{align}\label{e39}
v_{\Omega,c}(x)&=\int_{\Omega}dy\,\int_{\R^+}dt\, p_{\Omega,c}(x,y;t)\nonumber \\ &
=\int_{\Omega}dy\,\int_{\R^+}dt\, e^{-ct}p_{\Omega}(x,y;t)\nonumber \\ &
\ge \int_{\Omega}dy\,\int_0^Tdt\, e^{-ct}p_{\Omega}(x,y;t)\nonumber \\ &
\ge e^{-cT}\int_{\Omega}dy\,\int_0^Tdt\, p_{\Omega}(x,y;t)\nonumber \\ &
=e^{-cT}\Big(v_{\Omega}(x)-\int_T^{\infty}dt\,\int_{\Omega}dy\,p_{\Omega}(x,y;t)\Big).
\end{align}
The double integral in the right-hand side of \eqref{e39} is estimated using the heat semigroup property and Tonelli's Theorem,
\begin{align}\label{e40}
\int_T^{\infty}dt\,\int_{\Omega}dy\,p_{\Omega}(x,y;t)&=\int_T^{\infty}dt\,\int_{\Omega}dy\,\int_{\Omega}dz\,p_{\Omega}(x,z;t/2)p_{\Omega}(z,y;t/2)\nonumber \\ &
\le\int_T^{\infty}dt\,\int_{\Omega}dz\,p_{\Omega}(x,z;t/2)\int_{\Omega}dy\,p_{\Omega}(z,y;t/2).
\end{align}
Lemma 3 in \cite{vdBC} asserts that
\begin{equation}\label{e41}
p_{\Om}(z,y;t)\le (4\pi t)^{-m/2}2^{m/4}e^{-t\lambda_1(\Om)-|z-y|^2/(8t)}.
\end{equation}
This gives
\begin{align}\label{e42}
\int_{\Om}dy\,p_{\Om}(z,y;t)&\le (4\pi t)^{-m/2}2^{m/4}\int_{\R^m}dy\,e^{-t\lambda_1(\Om)/4-|z-y|^2/(8t)}\nonumber \\ &
=2^{3m/4}e^{-t\lambda_1(\Om)/4}.
\end{align}
By \eqref{e40} and \eqref{e42},
\begin{align}\label{e43}
\int_T^{\infty}dt\,\int_{\Omega}dy\,p_{\Omega}(x,y;t)&\le2^{3m/4}\int_T^{\infty}dt\,e^{-t\lambda_1(\Om)/8}\int_{\Omega}dz\,p_{\Omega}(x,z;t/2)\nonumber \\ &
\le 2^{3m/4}e^{-T\lambda_1(\Om)/8}\int_{\R^+}dt\,\int_{\Omega}dz\,p_{\Omega}(x,z;t/2)\nonumber \\ &
=2^{(4+3m)/4}e^{-T\lambda_1(\Om)/8}v_{\Om}(x).
\end{align}
By combining \eqref{e39} and \eqref{e43} we find
\begin{equation}\label{e47}
v_{\Omega,c}(x)\ge e^{-cT}\big(1-2^{(4+3m)/4}e^{-T\lambda_1(\Omega)/8}\big)v_{\Omega}(x).
\end{equation}
Choosing $T$ as to maximise the right-hand side of \eqref{e47} gives that
\begin{equation}\label{e48}
T=\frac{8}{\lambda_1(\Omega)}\log\Big(2^{(4+3m)/4}\Big(1+\frac{\lambda_1(\Omega)}{8c}\Big)\Big).
\end{equation}
Inequality \eqref{e36} then follows by \eqref{e47} and \eqref{e48}. The inequalities \eqref{e37} and \eqref{e38} follow immediately from \eqref{e36}.
\end{proof}

\noindent{\it Proof of Theorem \textup{\ref{the1}.}}
(i) Since $0\le V\le c$ we have by \eqref{e35},
\begin{equation}\label{e56}
\| v_{\Omega,V} \|_1  \le \| v_{\Omega} \|_1,
\end{equation}
and
\begin{equation}\label{e57}
\|v_{\Omega,V}\|_{\infty}\ge \|v_{\Omega,c}\|_{\infty}.
\end{equation}
The upper bound \eqref{e7} follows from \eqref{e56}, \eqref{e57}, and \eqref{e38}.

Similarly, by \eqref{e35}, we have
\begin{equation}\label{e58}
\|v_{\Omega,V}\|_{\infty}\le \|v_{\Omega}\|_{\infty},
\end{equation}
and
\begin{align}\label{e59}
\|v_{\Omega,V}\|_1\ge\|v_{\Omega,c}\|_1.
\end{align}
The lower bound \eqref{e8} follows from \eqref{e58}, \eqref{e59}, and \eqref{e37}.

\medskip

\noindent{\textup{(ii)}}
By using the inequality
\begin{equation*}%\label{e61}
\big(1+\theta^{-1}\big)^{\theta}\le e, \qquad \forall\theta>0,
\end{equation*}
with $\theta=8\Vert V_n\Vert_{\infty}/{\lambda_1(\Omega_n)}$, we obtain by \eqref{e7},
\begin{equation*}%\label{e62}
\Phi(\Omega_n,V_n)\le e2^{2(3m+4)\eta}(1+8\eta)\Phi(\Omega_n),
\end{equation*}
and by \eqref{e8},
\begin{equation*}%\label{e63}
\Phi(\Omega_n,V_n)\ge e^{-1}2^{-2(3m+4)\eta}(1+8\eta)^{-1}\Phi(\Omega_n).
\end{equation*}

\medskip

\noindent{\textup{(iii)}}
To verify \eqref{e50} we have by Lemma 4 in \cite{MvdBS},
$$\int_{\Omega}dy\, p_{\Omega}(x,y;t)\ge 1-2^{(m+2)/2}e^{-d_{\Omega}(x)^2/(8t)}.$$
Hence we find by \eqref{e24}
\begin{equation}\label{e53}
\int_{\Omega}dx\,\int_{\Omega}dy\, p_{\Omega}(x,y;t)\ge |\Omega|-2^{(m+2)/2}\int_{\Omega}dx\,e^{-d_{\Omega}(x)^2/(8t)}.
\end{equation}
Multiplying both sides of \eqref{e53} by $e^{-ct}$ and integrating with respect to $t$ gives
\begin{align}\label{e54}
\| v_{\Omega,c} \|_1&\ge \frac{|\Omega|}{c}-2^{(m+2)/2}\int_{\Omega}dx\,\int_{\R^+}dt\, e^{-ct-d_{\Omega}(x)^2/(8t)}\nonumber \\ &
\ge \frac{|\Omega|}{c}-2^{(m+2)/2}\int_{\Omega}dx\,\int_{\R^+}dt\, e^{-ct/2}\sup\{e^{-ct/2-d_{\Omega}(x)^2/(8t)}:t>0\}\nonumber \\ &
=\frac{|\Omega|}{c}-\frac{2^{(m+4)/2}}{c}\int_{\Omega}dx\,e^{-c^{1/2}d_{\Omega}(x)/2}.
\end{align}
On the other hand using \eqref{e42},
\begin{align}\label{e55}
\| v_{\Omega,c} \|_\infty &=\sup_{x\in\Omega}\int_{\R^+}dt\, e^{-ct}\int_{\Omega}dy\, p_{\Omega}(x,y;t)\nonumber \\ &
\le\sup_{x\in\Omega}\int_{\R^+}dt\, e^{-ct}\int_{\R^m}dy\, p_{\R^m}(x,y;t)\nonumber \\ &
=\frac{1}{c}.
\end{align}
The second inequality in \eqref{e50} follows from \eqref{e54} and \eqref{e55}. Finally, \eqref{e51} follows from \eqref{e50}, and Lebesgue's dominated convergence theorem,
while \eqref{e52} follows from \eqref{e8} and (i).
\hfill  $\square$

\medskip

\medskip

\section{Proof of Theorem \ref{the7} }\label{sec4}

\medskip

\noindent{\it Proof of Theorem \textup{\ref{the7}}.} The proof follows a method from \cite{HLP}, which can be summarised as follows: construct a  measure $\mu$  so that
 the efficiency of the eigenfunction of the first eigenvalue associated to $- \Delta + \mu$ almost equals $1$
 and then approximate  $\mu$  in the sense
of $\gamma$-convergence by a sequence of domains.
We refer the reader to \cite[Definition 4.8]{DMM87} (see also \cite[Chapter 4]{BB05}) for the notions of
$\gamma$-convergence, relaxed Dirichlet problems and approximations by sequences of domains. Recall that a measure on a domain is said to be capacitary if it is nonnegative, not necessarily finite and Borel, and which is in addition absolutely continuous with respect to the capacity. One can then define a relaxed Dirichlet problem (see \cite[Definition 3.1]{DMM87}, and \cite[ Sections 4.3 and  3.6]{BB05}). The equations $-\Delta v +\mu v=1$ and $-\Delta v+\mu v =\lambda v$ can then be solved in the weak sense. The sequence of sets $\Om_n\subset B_1$ is said to $\gamma$-convergence to the capacitary measure $\mu$ if the sequence of weak solutions  $v_n \in H^1_0(\Om_n)$  of $-\Delta v_n=1$ converges strongly in $L^2(B_1)$ to the weak solution $v \in H^1_0(B_1)\cap L^2(\mu)$ of $-\Delta v +\mu v=1$. As a consequence of the $\gamma$-convergence, the sequence of eigenvalues and corresponding eigenfunctions on the moving domain $(\Omega_n)$ converge in a suitable sense
 The $\gamma$-convergence is metrisable.

Let $\ov B_r$ denote the closure of $B_r$.
For any $0 < \varepsilon < 1$,
we consider the Dirichlet-Neumann eigenvalue problem on the annulus $A_\vps =B_1\sm \ov B_{1-\vps}$. Denote by $\lambda_\vps$ the first eigenvalue
and by $u_\vps$ a corresponding eigenfunction,

\begin{equation*}
\begin{cases}
-\Delta u_\vps =\lambda_\vps u_\vps,&\text{in }A_\vps,\\
{ u_\vps}=0,&\text{on }\partial B_{1},\\
\frac{\partial { u_\vps}}{\partial \nu}= 0,&\text{on }\partial B_{1-\vps},
\end{cases}
\end{equation*}
where $\nu$ denotes the inward-pointing normal on the sphere $\partial B_{1-\vps}$.
One can show that $\lb_\vps$ is simple and strictly positive  and that $u_\vps$ is radially symmetric and has a constant sign, say positive.
In particular, the restriction of $u_\vps$ to $\partial B_{1-\vps}$ equals a positive constant, $c_\vps > 0$.
We continuously extend $u_\vps$ inside $B_{1-\vps}$ by  $c_\vps$ and denote the resulting function, defined on $\ov B_1$, by $v_\vps$.

\medskip

\medskip

Since the normal derivatives of $v_\vps$ on both sides of $\partial B_{1-\vps}$ vanish, $\Delta  v_\vps$ is an $L^2$-function. More precisely, one has
$$
-\Delta v_\vps= \lb_\vps u_\vps 1_{A_\vps}, \quad \mbox  {in } {\mathcal D}'(B_1),
$$
where ${\mathcal D}'(B_1)$ denotes the space of distributions on $B_1$.
By adding on both sides the $L^2$ function $ \lb_\vps v_\vps 1_{\ov B_{1-\vps}}$ we get \\
$$-\Delta v_\vps+  \lb_\vps v_\vps 1_{\ov B_{1-\vps}}= \lb_\vps v_\vps,\quad \mbox  {in } {\mathcal D}'(B_1).$$
We view $\mu :=  \lb_\vps 1_{\ov B_{1-\vps}}$ as a capacitary measure on $B_1$. Then formally, we get
$$
-\Delta v_\vps+ \mu  v_\vps = \lb_\vps v_\vps,\quad \mbox  {in } {\mathcal D}'(B_1).
$$
Combined with the fact that $v_\vps >0$, this means that $\lb_\vps$ is the first Dirichlet eigenvalue of $- \Delta + \mu$.
We assume from now on that $v_\vps$ is normalised by $\|v_\vps\|_{L^2(B_1)}=1$.
Since the measure $\mu$ is finite with support in $B_1$, the first eigenvalue $\lambda_\vps$ is simple.

In view of the Dal Maso-Mosco density result \cite[Theorem 4.16]{DMM87}, there exists a sequence $\Om_n\sq B_1$, $n \ge 1$, such that   $\Om_n $ $\gamma$-converges to $\mu$. We can assume that
the boundary $\partial \Om_n$ of $\Om_n$ is smooth.
Indeed, otherwise we can replace each $\Om_n$ by an inner approximation with a smooth open set, which is close enough in  the sense of the distance associated to the $\gamma$-convergence. Furthermore, we can also assume that $\Om_n$ is connected. Indeed, since $\Om_n$ is smooth, it has a finite number of smooth connected components, which are separated from each other by a positive distance. These components can be joined by a finite number of thin tubes, connecting the set $\Om_n$. As the width of the tubes vanishes, the sequence $\gamma$-converges by Sverak's theorem (see \cite[Theorem 4.7.1]{BB05}).

It is possible to explicitly construct such a sequence $(\Omega_n)$ in the spirit of Cioranescu-Murat \cite{CM}, but such a construction is not needed for the rest of the proof.
The only fact we need to keep in mind is that $\Om_n \sq B_1$, so that $|\Om_n| \le |B_1|$.
Denote by $\lambda_1(\Omega_n)$ the first Dirichlet eigenvalue of $-\Delta$ on $\Omega_n$ and by $u_n$ the $L^2-$normalised, positive eigenfunction corresponding to $\lambda_1(\Omega_n)$.
We extend $u_n$ to $ B_1$ by setting it to $0$ on  $B_1 \setminus \ov \Omega_n$ and by a slight abuse of notation, denote this extension again by $u_n$.
The $\gamma$-convergence of $(\Omega_n)$, together with the compact embedding of $H^1_0(B_1)$ in $L^1(B_1)$,  imply that (i) $\lb_1(\Om_n) \rightarrow \lb_\vps$, and
(ii) $u_n\rightharpoonup v_\vps$ weakly in $H^1_0(B_1)$ and strongly in $L^1(B_1)$.
Assuming  that
\begin{equation}\label{L infty}
\lim_{n \to \infty }\|u_n\|_\infty = \|v_\vps\|_\infty.
\end{equation}
 we get
$$
\liminf_{n \rightarrow +\infty} \frac{\int_{\Om_n} u_n }{\|u_n\|_\infty |\Om_n|} \ge  \frac{\int_{B_1} v_\vps }{\|v_\vps\|_\infty |B_1|}.
$$
Since the right-hand side is arbitrarily close to $1$ when $\vps \downarrow 0$, the proof of Theorem \ref{the7} is then completed by a diagonal selection procedure.

It remains to show \eqref{L infty}. This kind of assertion is known to be true in a general setup. In essence it is a consequence
of the subharmonicity of the eigenfunctions $u_n$, $n \ge 1$. (For a similar result for the torsion function see \cite[Theorem 2.2]{HLP}.)
For the sake of completeness, we give below a proof. It slightly differs from the one in \cite[Theorem 2.2]{HLP}.

First note that since $u_n \to v_\vps$ strongly in $L^1$, it follows that
$$
\liminf _{n \to +\infty} \|u_n\|_\infty  \ge  \|v_\vps\|_\infty .
$$
To prove that $ \limsup _{n \to +\infty} \|u_n\|_\infty  \le  \|v_\vps\|_\infty $ we argue as follows.
Being convergent,  the sequence $(\lb_1(\Omega_n))$ is bounded, and so is $\| u_n \|_\infty$.
Choose $M>0$ so that  for any $n \in \N$
$$\lb_1(\Omega_n) u_n(x)  \le M, \quad \forall \ x \in B_1,$$
and therefore
$$
-\Delta u_n(x)  \le M ,  \qquad \forall \ x \in B_1 .
$$
Let $x_n \in B_1$ be a maximum point for $u_n$. By taking, if necessary, a subsequence we may assume that $x_n \to x^*$.
Furthermore,
$$-\Delta u_n \le M= M \Delta \frac{|x-x_n|^2}{2m},  \quad \mbox  {in } {\mathcal D}'( \R^m),$$
or
$$-\Delta \Big (u_n + M  \frac{|x-x_n|^2}{2m}\Big)\le 0, \quad \mbox  {in } {\mathcal D}'( \R^m).$$
By the subharmonicity of the function $x \mapsto u_n(x) + M  \frac{|x-x_n|^2}{2m}$ around $x_n$, it then follows that for
$\delta >0$ sufficiently small,
$$\|u_n \|_\infty = u_n(x_n) \le \frac{\int _{B(x_n; \delta)}dx \big( u_n(x) + M  \frac{|x-x_n|^2}{2m}\big) }{|B(x_n; \delta)|}.$$
Taking the limit $n \to +\infty$ we obtain
$$\limsup _{n \to +\infty} \|u_n\|_\infty \le \frac{\int _{B(x^*; \delta)}dx\big(v_\vps(x) + M  \frac{|x-x^*|^2}{2m}\big)}{|B(x^*;\delta)|},$$
or
$$\limsup _{n \to +\infty} \|u_n\|_\infty \le \frac{\|v_\vps\|_\infty |B(x^*; \delta)| + M  \frac{\delta^2}{2m} |B(x^*; \delta)|}{|B(x^*; \delta)|}
= \|v_\vps\|_\infty + M  \frac{\delta^2}{2m}.$$
Letting $\delta \downarrow 0$, completes the proof.
\hfill$\square$

\bigskip

\begin{remark}\label{rem0}
If $\Omega$ is an open connected subset of $\R^m$ with $m \ge 2$ and $0<|\Omega|<\infty$, then
\begin{enumerate}
\item[\textup{(i)}]
\begin{equation}\label{p1}
\Phi(\Omega)\ge \frac{E(\Omega)}{1+k_m\lambda_1(\Omega)^{m/4}|\Omega|^{1/2}\big(1-E(\Omega)\big)^{1/2}},\hspace{14mm} m=2,3,
\end{equation}
\item[\textup{(ii)}]
\begin{equation}\label{p2}
\Phi(\Omega)\ge \frac{E(\Omega)}{1+k_m\lambda_1(\Omega)|\Omega|^{2/m}\big(1-E(\Omega)\big)^{1/(m-1)}},\qquad m\ge 4,
\end{equation}
\end{enumerate}
where
\begin{equation}\label{p0}
k_m=
\begin{cases}
2(8\pi)^{-m/4}\Gamma((4-m)/4),\hspace {38mm} m=2,3,\\
\pi^{-1}(m-2)^{-1}m^{-1/(m-1)}\big(\Gamma((m+2)/2)\big)^{2/m},\qquad m\ge 4.
\end{cases}
\end{equation}
\end{remark}

\begin{proof}
Putting $V=0$ in \eqref{e24} one obtains by using  \eqref{e41},
\begin{align}\label{p3}
v_{\Omega}(x)&\ge \int_{\Omega} dy \int_{\R^+} dt\, p_{\Omega}(x,y;t)\frac{\varphi_{1,\Omega}(y)}{\|\varphi_{1, \Omega}\|_{\infty}}\nonumber \\ &
=\lambda_{1}(\Omega)^{-1}\frac{\varphi_{1,\Omega}(x)}{\|\varphi_{1, \Omega}\|_{\infty}}.
\end{align}
Integrating both sides of \eqref{p3} yields,
\begin{equation}\label{p4}
\|v_{\Omega}\|_1\ge \lambda_{1}(\Omega)^{-1}\frac{\|\varphi_{1,\Omega}\|_1}{\|\varphi_{1, \Omega}\|_{\infty}}.
\end{equation}
To obtain the stated lower bound for $\Phi(\Omega)$ it remains to find an upper bound for $\|v_{\Omega}\|_{\infty}$.

First consider case (i).
By \eqref{p3}, the Cauchy-Schwarz inequality, and the heat semigroup property, one sees that
\begin{align}\label{p5}
v_{\Omega}(x)-\lambda_{1}&(\Omega)^{-1}\frac{\varphi_{1,\Omega}(x)}{\|\varphi_{1, \Omega}\|_{\infty}}
=\int_{\Omega} dy \int_{\R^+} dt\, p_{\Omega}(x,y;t)\bigg(1-\frac{\varphi_{1,\Omega}(y)}{\|\varphi_{1, \Omega}\|_{\infty}}\bigg)\nonumber \\ &
\le \int_{\R^+} dt\, \bigg(\int_{\Omega} dy(p_{\Omega}(x,y;t))^2\bigg)^{1/2}\bigg(\int_{\Omega} dy\bigg(1-\frac{\varphi_{1,\Omega}(y)}{\|\varphi_{1, \Omega}\|_{\infty}}\bigg)\bigg)^{1/2}\nonumber \\ &
=|\Omega|^{1/2}\int_{\R^+} dt\,p_{\Omega}(x,x;2t)^{1/2}\big(1-E(\Omega)\big)^{1/2}.
\end{align}
Choosing $\beta=\frac12$ in Lemma 1 of \cite{vdBC} gives by domain monotonicity of the Dirichlet heat kernel, and \eqref{e42}
\begin{equation}\label{p6}
p_{\Omega}(x,x;2t)\le e^{-t\lambda_1(\Omega)}(4\pi t)^{-m/2}.
\end{equation}
Substitution of \eqref{p6} into the right-hand side of \eqref{p5}, evaluating the resulting integral with respect to $t$, and taking the supremum over all $x\in\Omega$ gives
\begin{equation}\label{p7}
\|v_{\Omega}\|_{\infty}-\lambda_{1}(\Omega)^{-1}\le 2(8\pi)^{-m/4}|\Omega|^{1/2}\Gamma((4-m)/4)\lambda_1(\Omega)^{-1 + m/4}\big(1-E(\Omega)\big)^{1/2}.
\end{equation}
Inequality \eqref{p1} follows from \eqref{p4}, \eqref{p7} with the values for $k_2$ and $k_3$ given in \eqref{p0}.

Next consider case (ii). By the first equality in \eqref{p5} we have by domain monotonicity of the Dirichlet heat kernel
\begin{align*}%\label{p8}
v_{\Omega}(x)-\lambda_{1}&(\Omega)^{-1}\frac{\varphi_{1,\Omega}(x)}{\|\varphi_{1, \Omega}\|_{\infty}}
\le \int_{\Omega} dy \int_{\R^+} dt\, p_{\R^m}(x,y;t)\bigg(1-\frac{\varphi_{1,\Omega}(y)}{\|\varphi_{1, \Omega}\|_{\infty}}\bigg) \nonumber \\ &
=c_m\int_{\Omega} dy |x-y|^{2-m}\bigg(1-\frac{\varphi_{1,\Omega}(y)}{\|\varphi_{1, \Omega}\|_{\infty}}\bigg),
\end{align*}
where
\begin{equation*}%\label{p9}
c_m=\frac{\Gamma((m-2)/2)}{4\pi^{m/2}}.
\end{equation*}
By H\"older's inequality with exponents $p=\frac{m-1}{m-2}$ and $q=m-1$ we have
\begin{align*}%\label{p10}
v_{\Omega}(x)-&\lambda_{1}(\Omega)^{-1}\frac{\varphi_{1,\Omega}(x)}{\|\varphi_{1,\Omega}\|_{\infty}}\nonumber \\ &
\le c_m\bigg(\int_{\Omega}\frac{dy}{|x-y|^{m-1}}\bigg)^{(m-2)/(m-1)}\bigg( \int_\Omega dy \big(1-\frac{\varphi_{1,\Omega}(y)}{\|\varphi_{1, \Omega}\|_{\infty}} )\bigg)^{1/(m-1)}\nonumber \\ &
\le c_m\bigg(\int_{\Omega^*}\frac{dy}{|y|^{m-1}}\bigg)^{(m-2)/(m-1)}\bigg(\int_\Omega dy \big(1-\frac{\varphi_{1,\Omega}(y)}{\|\varphi_{1, \Omega}\|_{\infty}} \big)\bigg)^{1/(m-1)}\nonumber \\ &
=c_m \big(m\omega_m R^*\big)^{(m-2)/(m-1)}|\Omega|^{1/(m-1)}\big(1-E(\Omega)\big)^{1/(m-1)},
\end{align*}
where we have used Schwarz symmetrisation with $\Omega^*=B_{R^*}$, and $\omega_m (R^*)^m=|\Omega|.$

Taking the supremum over all $x\in \Omega$ and using the formulae for $c_m$ and $R^*$ gives,
\begin{align*}%\label{p14}
\|v_{\Omega}\|_{\infty}\le\lambda_{1}(\Omega)^{-1}&+\pi^{-1}(m-2)^{-1}m^{-1/(m-1)}\big(\Gamma((m+2)/2)\big)^{2/m}\nonumber \\ &\times|\Omega|^{2/m}\big(1-E(\Omega)\big)^{1/(m-1)}.
\end{align*}
This, together with \eqref{p4}, implies the assertion for $m\ge 4$.
\end{proof}
\medskip

We see from the proof of Remark \ref{rem0}(ii) that the case $m=3$ could also have been included.
However, that would have given $\lambda_1(\Omega)|\Omega|^{2/3}$ in the denominator.
By Theorem \ref{the9} (iii) we have, for $m=3$, that $\lambda_1(\Omega)|\Omega|^{2/3}\gg1$ if $E(\Omega)$ is close to $1$.
Then $\lambda_1(\Omega)^{3/4}|\Omega|^{1/2}\ll\lambda_1(\Omega)|\Omega|^{2/3},$ and so \eqref{p1} gives a better bound in that case.
However, bounds \eqref{p1} and \eqref{p2} do not imply that if $\Phi(\Omega)$ is close to $1$ then $E(\Omega)$ is close to $1$ since, by Theorem \ref{the9} (iii),
$\lambda_1(\Omega)|\Omega|^{2/m}$ becomes large.

\medskip

\medskip

\section{Proof of Theorem \ref{the9} }\label{sec5}

\medskip

\noindent
{\em  Proof of Theorem \textup{\ref{the9}(i)}}.
Since $\|u\|_{\infty}>0$ we can re-scale both $u$ and $\Omega$,
such that $\|u\|_{\infty}=1$, and $|\Omega|=\omega_m$. Inequality \eqref{p22}
then reads
\begin{equation}\label{p25}
\bigg( 1 - \fint_{\Omega}u\bigg)\int_{\Omega}|\nabla u|^2\ge \frac{4m^2}{(m+2)^2}\omega_m,
\end{equation}
with $|\Omega|=\omega_m, \,  \|u\|_\infty = 1,$ and  $\fint_{\Omega}u\ge \frac{2}{m+2}.$
Note that replacing $u$ by its positive part $u^{+}$ decreases the left-hand side of \eqref{p25}, and furthermore,
$u^{+}\le 1,\,\fint_{\Omega} u^+ \ge \frac{2}{m+2}$.
So it suffices to prove that for any $m \ge 2,$
\begin{equation*}%\label{p28}
(1-\theta) F(\theta) \ge \frac{4m^2}{(m+2)^2}\omega_m \, , \qquad  \forall \, \theta\in [2/(m+2),1),
\end{equation*}
where
$$
F(\theta):= \inf\big\{\int_{\Omega}|\nabla u|^2: \, u\in H_0^{1}(\Omega), 0 \le u \le 1,\fint_{\Omega}u=\theta \big\} \, .
$$

We make some preliminary observations. By Schwarz rearrangement we may consider the infimum in the definition of $F$ over the collection $H_0^{*1}(B_1)$ of all radially symmetric, decreasing functions $u$ in $H_0^1(B_1)$
since this rearrangement decreases the energy and leaves the other constraints unchanged.
So,
$$
F(\theta)  \ge  F^*(\theta), \quad 2/(m+2) \le \theta < 1,
$$
where
\begin{equation}\label{minimization 2}
 F^*(\theta) =   \inf\big\{\int_{B_1}|\nabla u|^2 : \, u \in H_0^{*1}(B_1), 0 \le u \le 1,\,\fint_{B_1}u=\theta \big\}.
\end{equation}
First note that
$$
  \inf\big\{\int_{B_1}|\nabla u|^2 : \, u \in H_0^{*1}(B_1), \,  \fint_{B_1}u=\theta \big\},
$$
admits a minimiser. By the Lagrange multiplier theorem, there exists a constant $c$ such that
$$
- \Delta u = c\, , \qquad u \in H^{*1}_0(B_1)\, , \qquad \fint_{B_1} u = \theta,
 $$
Since $u$ is radially symmetric and decreasing,
in the sequel, by a slight abuse of notation, we write $u(r)$ instead of $u(x)$.
By a straightforward computation one sees that for any $2/(m+2) \le \theta < 1$
\begin{equation}\label{unconstrained}
u(r) = 2^{-1}(m+2)\theta(1 - r^2).
\end{equation}
In particular, $\|u \|_\infty = u(0)\ge 1$.
Note that we could have written $0\le u,\, \| u\|_\infty =1$ instead of $0\le u\le 1$ in the right-hand side of \eqref{minimization 2}.

For any $2/(m+2) \le \theta <1,$  $ F^*(\theta)$  admits a minimiser.
Since the Dirichlet energy is strictly convex, it is unique, and we denote it by $u_\theta$.
Let
\begin{equation*}%\label{minimization}
f(\theta) = (1- \theta)F^*(\theta).
\end{equation*}
Since $F(\theta) \ge F^*(\theta)$, it suffices to show that
\begin{equation}\label{to be proved}
f(\theta) \ge \frac{4m^2}{(m+2)^2}\omega_m\, , \qquad \ \forall  \theta \in [2/(m+2) , 1)\, .
\end{equation}
If $\theta =  \frac{2}{m+2}$, then by \eqref{unconstrained},  $u_{2/(m+2)} (r) = 1- r^2$. Hence
\begin{equation}\label{minimum}
 f(2/(m+2)) = \Big(1- \frac{2}{m+2}\Big) \int_{B_1} (-2r)^2 mr^{m-1} \omega_m dr = \frac{ 4m^2}{(m+2)^2} \omega_m.
\end{equation}
Hence \eqref{to be proved} is satisfied for $\theta =  \frac{2}{m+2}$.

\medskip

The remaining part of the proof consists of four parts.
In part (a) we show that \eqref{to be proved} holds  for  any $m\ge 2$ and $\theta\in [\frac{2}{m+2},\theta_m^*]$  where
\begin{equation}\label{p27}
\theta_m^* :=\frac{m+\big(m^2+8m\big)^{1/2}}{2(m+2)}.
\end{equation}
We note that $2/(m+2) <  2/3 < \theta_m^* < 1$ for any $m \ge 2$.
In part (b) we prove that \eqref{to be proved} holds for any $m\ge 4$ and $\theta\in [\theta_m^*,1)$.
In part (c) we show that \eqref{to be proved} holds for $m=2,3$,  using the Euler-Lagrange equation of a variational problem, related
to an obstacle problem. See \eqref{np1} below.
Finally, in part (d) we verify that {\em equality} in \eqref{p22} holds if and only if $\Omega$ is a ball and $u$ is a multiple of the torsion function for that ball.
This completes the proof of Theorem \ref{the9} (i).

\medskip

(a) Assume that $\frac{2}{m+2} \le \theta \le \theta_m^*.$
In a first step we prove that $\theta\mapsto f(\theta)$ is increasing on $[\frac{2}{m+2},\frac23]$ which by \eqref{minimum}
implies that \eqref{to be proved} holds for $\theta$ in the latter interval.
Given $\frac{2}{m+2} < \theta < 1$, choose any $\varepsilon>0$ with the property
\begin{equation*}%\label{p50}
\frac{\theta}{1+\varepsilon}>\frac{2}{m+2}.
\end{equation*}
Let $\overline{u}_\theta$ be the minimiser of the obstacle problem
\begin{equation*}%\label{p51}
\inf\{\int_{B_1}|\nabla u|^2: \, \, u\in  H_0^{*1}(B_1), \, 0\le u \le 1+\varepsilon, \, \fint_{B_1} u=\theta\}.
\end{equation*}
Then, by inclusion of the class of admissible test functions,
\begin{equation}\label{p52}
\int_{B_1}|\nabla\overline {u}_\theta|^2\le \int_{B_1}|\nabla u_{\theta}|^2,
\end{equation}
where $u_{\theta}$ is the minimiser of $F^*(\theta)$.
Furthermore $(1+\varepsilon)^{-1}\overline{u}_\theta$ is the minimizer $u_{\theta/(1+\varepsilon)}$ for $F^*( \theta/(1 + \varepsilon))$,
since  $0 \le (1+\varepsilon)^{-1}\overline{u}_\theta \le 1$, and
$\fint_{B_1}\frac{\overline{u}_\theta}{1+\varepsilon}=\frac{\theta}{1+\varepsilon}$.
We wish to prove that
\begin{equation}\label{p53}
f(\theta/(1 + \varepsilon)) = \Big(1- \frac{\theta}{ 1 + \varepsilon}\Big) \int_{B_1}\frac{|\nabla \overline{ u}_\theta |^2}{ (1 + \varepsilon)^2} \le
(1-\theta) \int_{B_1}|\nabla u_{\theta}|^2 = f(\theta),
\end{equation}
By \eqref{p52}, inequality \eqref{p53} holds if
\begin{equation*}%\label{p55}
\frac{1}{(1+\varepsilon)^2}\bigg(1-\frac{\theta}{1+\varepsilon}\bigg)\le(1-\theta),
\end{equation*}
or, after simplifying,
\begin{equation}\label{p56}
1\le (3+3\varepsilon +\varepsilon^2)(1-\theta).
\end{equation}
The latter inequality clearly holds for any $\varepsilon \ge 0$ if $ \frac{2}{m+2} \le \theta\le \frac23$.
Thus we have proved that $f(\theta)$ is increasing on the interval $[\frac{2}{m+2},\frac23]$.

By the same argument we now prove that $f(\theta) \ge f(2/(m+2))= \frac{ 4m^2}{(m+2)^2}\omega_m $ also for $\theta \in ( \frac23, \theta_m^*]$.
(However, we do not prove that $f$ is increasing on the interval $( \frac23, \theta_m^*]$.)
Given $\theta \in ( \frac23, \theta_m^*]$, we want to find $\varepsilon>0$ so that
\begin{equation}\label{p57}
 \frac{2}{m+2}  \le  \frac{\theta}{1+\varepsilon}\le \frac 23 ,\qquad 1\le \big( 3+3\varepsilon+\varepsilon^2\big) (1-\theta).
\end{equation}
By \eqref{p53} and  \eqref{p56}, one then infers that
$$f(\theta) \ge f(\theta/(1+\varepsilon)) \ge f(2/(m+2)).$$
To find $\varepsilon > 0$ so that  \eqref{p57} holds, we solve
\begin{equation}\label{eq for epsilon}
1= 3(1-\theta)+3\varepsilon(1-\theta)+\varepsilon^2(1-\theta)
\end{equation}
for $\varepsilon$, and verify that for the given range of $\theta,$ one has  $\frac{2}{m+2}  \le  \frac{\theta}{1+\varepsilon}\le \frac 23$.
The positive solution of \eqref{eq for epsilon} is given by
$$\varepsilon=-\frac32+ \frac12\bigg(\frac{1+3\theta}{1-\theta}\bigg)^{1/2}>0.$$
Since $$\theta\le \theta_m^*$$ we have that the requirement $\frac{2}{m+2} \le \frac{\theta}{1+\varepsilon}$ is fulfilled.
Using $\theta \ge \frac 23$, one sees that  $ \frac{\theta}{1+\varepsilon} \le \frac 23$.

\medskip

(b) In this part we assume that $m \ge 4$ and $\theta_m^* \le \theta < 1.$
Let $ 0 \le  r_0 < 1$ and let $u \in H_0^{*1}(B_1)$ with $0 \le u \le 1$ and $\theta = \fint_{B_1} u.$
Since $u$ is radially symmetric and decreasing,
\begin{equation*}%\label{p29}
\omega_m^{-1}\int_{B_1\setminus\overline{B}_{r_0}}u \le \, \omega_m^{-1}u(r_0) \, |B_1\setminus\overline{B}_{r_0}|=u(r_0)(1-r_0^m).
\end{equation*}
Since, $0 \le u \le 1$ and $\theta = \fint_{B_1} u,$ we conclude that
\begin{equation*}%\label{p30}
 r_0^m\ge \omega_m^{-1}\int_{B_{r_0}}u = \omega_m^{-1}\int_{B_1}u - \omega_m^{-1}\int_{B_1 \setminus  \ov B_{r_0}}u
 \ge \theta-u(r_0)(1-r_0^m).
\end{equation*}
It then follows that
\begin{equation}\label{p31}
u(r_0)\ge  \eta_m(\theta, r_0) := \frac{ \theta-r_0^m}{1-r_0^m}  \ge 0 \, , \qquad  \forall \, r_0\in [0,\theta^{1/m}].
\end{equation}
By inclusion of the admissible test functions one has
\begin{align}\label{p32}
f(\theta)\ge (1-\theta)\inf\big\{\int_{B_1}|\nabla u|^2 : \, u \in H_0^{*1}(B_1), \, u(r_0)\ge \eta_m(\theta, r_0) \big\}.
\end{align}
The infimum in the right-hand side of \eqref{p32} is attained and its minimiser $u^*$ is given by
\begin{equation*}%\label{p33}
u^*(r)=
\begin{cases}
\eta_m(\theta, r_0),\hspace {18mm} 0\le r\le r_0,\\
\frac{1-r^{2-m}}{1-r_0^{2-m}}\eta_m(\theta, r_0),\qquad r_0\le r\le 1.
\end{cases}
\end{equation*}
A straightforward calculation gives
\begin{equation}\label{p34}
\int_{B_1}|\nabla u^*|^2=m(m-2)\omega_m\frac{\eta_m^2(\theta, r_0)}{r_0^{2-m}-1}\, .
\end{equation}
We now choose
\begin{equation}\label{p35}
r_0=\theta^{c/m},
\end{equation}
where $c\ge 1$ is to be determined later. This choice satisfies the constraint $0 \le r_0 \le \theta^{1/m}$ in \eqref{p31}. By \eqref{p34} and \eqref{p35},
\begin{equation}\label{p36}
f(\theta)\ge m(m-2)\omega_m\bigg(\frac{\theta-\theta^c}{1-\theta^c}\bigg)^2\frac{1-\theta}{\theta^{c(2-m)/m}-1}.
\end{equation}
By L'H\^opital's rule,
\begin{equation}\label{p371}
\liminf_{\theta\uparrow 1}f(\theta) \,  \ge  \, m^2 \omega_m\frac{(c-1)^2}{c^3}.
\end{equation}
The right-hand side of \eqref{p371} is maximised for $c=3$. This choice yields,
\begin{equation*}%\label{p39}
\liminf_{\theta\uparrow 1}f(\theta)\ge \frac{4m^2}{27} \omega_m.
\end{equation*}
Note that
\begin{equation}\label{p40}
\frac{4m^2}{27} \omega_m \ge \frac{4m^2}{(m+2)^2}\omega_m.
\end{equation}
if and only if $m\ge 4$. This is why the proof of \eqref{to be proved}
for $m=2,3$ has been deferred to part (c). By \eqref{p36} we have for $c=3$,
\begin{align}\label{p38}
f(\theta)&\ge m(m-2)\omega_m\bigg(\frac{\theta(1+\theta)}{1+\theta+\theta^2}\bigg)^2\frac{1-\theta}{\theta^{3(2-m)/m}-1}\nonumber \\ &
\ge m(m-2)\omega_m\bigg(\frac{\theta(1+\theta)}{1+\theta+\theta^2}\bigg)^2\frac{1-\theta}{\theta^{-3}-1}\nonumber \\ &
=m(m-2)\omega_m\frac{\theta^5(1+\theta)^2}{(1+\theta+\theta^2)^3}.
\end{align}

To prove \eqref{to be proved} for $\theta\in [\theta_m^*,1)$ and $m \ge 4$,
it suffices, by \eqref{p38}, to show that
\begin{equation}\label{p41}
\frac{\theta^5(1+\theta)^2}{(1+\theta+\theta^2)^3}\ge \frac{4m}{(m-2)(m+2)^2},\qquad \forall \, \theta\in [\theta_m^*,1).
\end{equation}
First observe that the left-hand side of \eqref{p41} is a product of non-negative increasing functions, $\theta\mapsto\frac{\theta(1+\theta)}{1+\theta+\theta^2}$ and
$\theta\mapsto\frac{\theta^3}{1+\theta+\theta^2}$, and so is increasing. So if \eqref{p41} holds for $\theta=\theta_m^*$ then it holds on the interval $[\theta_m^*,1)$.
Furthermore by \eqref{p27},
\begin{equation*}%\label{p42}
\theta_m^*=1-\frac{8}{(m+2)(m+4+(m^2+ 8m)^{1/2})}.
\end{equation*}
Hence $(\theta_m^*)_m$ is a strictly increasing sequence. Since the right-hand side of \eqref{p41} is decreasing in $m$ we conclude that if \eqref{p41} holds for $m=m_1$ then it holds for all $m\ge m_1$.
It is straightforward to verify
\begin{equation*}%\label{p43}
\theta_6^*> \frac{15}{16},
\end{equation*}
and that \eqref{p41} holds for $\theta=\frac{15}{16}$, and $m=6$. Hence it follows that \eqref{p41} holds for all  $\theta \in [\theta_m^*, 1)$ with $m\ge 6$.

To complete the proof of part (b) it remains to treat the cases $m=4$ and $m=5$. We first consider the case $m=4$. One computes
\begin{equation*}%\label{p44}
\theta_4^*=\frac{1+\sqrt 3}{3},
\end{equation*}
and by the first inequality in \eqref{p38}, one gets
\begin{equation}\label{p45}
f(\theta)\ge 8\omega_4\bigg(\frac{\theta(1+\theta)}{1+\theta+\theta^2}\bigg)^2\bigg(\frac{1-\theta}{\theta^{-3/2}-1}\bigg).
\end{equation}
Since both $\theta\mapsto \frac{\theta(1+\theta)}{1+\theta+\theta^2}$ and $\theta\mapsto \frac{1-\theta}{\theta^{-3/2}-1}$ are non-negative increasing functions on the interval $[0, 1)$
the right-hand side of \eqref{p45} is increasing in $\theta$. Note that the right-hand side of \eqref{p40} equals $16\omega_4/9$. Hence  by \eqref{p45} it suffices to verify that
\begin{equation}\label{p46}
\bigg(\frac{\theta^*_4(1+\theta^*_4)}{1+\theta^*_4+\theta{^*_4}^2}\bigg)^2\bigg(\frac{1-\theta^*_4}{\theta{^*_4}^{-3/2}-1}\bigg)\ge \frac{2}{9}.
\end{equation}
Numerical evaluation of the left-hand side of \eqref{p46} yields
\begin{equation*}%\label{p47}
\bigg(\frac{\theta^*_4(1+\theta^*_4)}{1+\theta^*_4+\theta{^*_4}^2}\bigg)^2\bigg(\frac{1-\theta^*_4}{\theta{^*_4}^{-3/2}-1}\bigg)\ge .238,
\end{equation*}
which implies \eqref{p46}.

Finally we consider the case $m=5$. One computes that
\begin{equation}\label{p48}
\theta_5^*\ge \frac{13}{14},
\end{equation}
and by the first inequality in \eqref{p38},
\begin{equation*}%\label{p48}
f(\theta)\ge 15\omega_5\bigg(\frac{\theta(1+\theta)}{1+\theta+\theta^2}\bigg)^2\frac{1-\theta}{\theta^{-9/5}-1}.
\end{equation*}
Since both $\theta\mapsto \frac{\theta(1+\theta)}{1+\theta+\theta^2}$ and $\theta\mapsto \frac{1-\theta}{\theta^{-9/5}-1}$ are non-negative increasing functions
on the interval $[0,1)$, so is
the right-hand side of \eqref{p45} and it remains, by \eqref{p45}, to verify that
\begin{equation}\label{p49}
\bigg(\frac{\theta^*_5(1+\theta^*_5)}{1+\theta^*_5+\theta{^*_5}^2}\bigg)^2\bigg(\frac{1-\theta^*_5}{\theta{^*_5}^{-9/5}-1}\bigg)\ge \frac{20}{147}.
\end{equation}
By \eqref{p48}, the left-hand side of \eqref{p49} is bounded from below by $.206$ while the right hand side of \eqref{p49} is bounded from above by $.137$. This completes the proof of part (b).

\medskip

(c) In this part we treat the cases $m=2$ and $m=3$. We begin with some preliminary considerations.
We note that the minimisation problem \eqref{minimization 2} is related to a volume constraint obstacle problem in $B_1$:
we claim that there exist $c>0$ and $0 \le l < 1$, depending on $\theta$, so that
$u_\theta$ satisfies the following system of equations,
\begin{equation}\label{np1}
\begin{cases}
-\Delta u=c,&\text{in }B_1\setminus \ov B_l,\\
u=1,&\text{on } \ov B_l,\\
u= 0,&\text{on }\partial B_1,\\
\frac {\partial u}{\partial \nu} =0, &\text{on }\partial B_l,
\end{cases}
\end{equation}
where $\nu$ denotes the inward pointing normal on the sphere $\partial B_l$.
Indeed, since $u_\theta$ is radially symmetric, decreasing and since $u_\theta(0) = 1$ (see \eqref{minimization 2}) and $u_\theta(1) = 0$,
there exists  a maximal number $0 \le l \equiv l(\theta) < 1$ so that $u_\theta (r) =1$ for $0 \le r \le l$.
By the Lagrange multiplier theorem, there exists a constant $c >0$ so that $-\Delta u_\theta  =c$ on $B_1 \sm \ov B_l$ in the sense of distributions.
It then follows from \cite[Theorem 2]{GE} that $u_\theta$ is $C^{1, \alpha}$ on $\ov B_1$ which implies that $\frac {\partial u_\theta}{\partial \nu} =0 $ on $\partial B_l$.
These observations establish \eqref{np1}.
Note that both $c$ and $l$ are uniquely determined by $\theta$.

We claim that the map
$$
\mathfrak{b} :  [\frac{2}{m+2}, 1) \to  [0,1) , \,    \theta \mapsto l(\theta) \, ,
$$
is an increasing bijection.
To prove the latter assertion, we construct for any given $0 \le l < 1$ a unique
radially symmetric, decreasing solution $u(\cdot ; l)$ of \eqref{np1} and show that $\theta \equiv \theta(l) := \fint_{B_1} u(\cdot; l)$,
satisfies $2/(m+2) \le \theta <1$ with $\theta(0) = 2/(m+2)$ and $\lim_{l \to 1} \theta(l)= 1$.
%Here and in the sequel, by a slight abuse of notation, we write $u(r)$ instead of $u(x)$.
First we note that $c$ is uniquely determined by $l$ since
for any given $0 \le l < 1$, the solution of  \eqref{np1} is given by a formula. To obtain it, note that the general radially symmetric solution of
$- \Delta u = c$ on the annulus $B_1 \setminus \ov B_l$ is of the form
$$
u(r)=
\begin{cases}
-c\frac{r^2}{4}+a \ln(r) +b \qquad \quad  \mbox{ if } m=2 \, ,\\
-c \frac{r^2}{2m} -\frac{a}{(m-2)r^{m-2}} +b \quad \mbox{ if } m\ge3
\end{cases}
$$
for some real constants $a,$ $b,$ $c$.
The condition $\frac {\partial u}{\partial \nu}(l) =0 $ implies that $a= \frac{l^m c}{m}$
so that the boundary condition $u(1)=0$ leads to
\begin{equation}\label{np2}
u(r)=
\begin{cases}
\frac{c}{4}( 1 - r^2) + \frac{l^2 c}{2}  \ln(r)  \qquad \qquad \quad \quad \mbox{ if } m=2,\\
 \frac{c}{2m}(1-r^2) + \frac{l^m c}{m(m-2)} (1-\frac{1}{r^{m-2}} ) \ \   \mbox{ if } m\ge3.
\end{cases}
\end{equation}
The value of $c$ is now obtained by the requirement $u(l)=1$,
\begin{equation}\label{np3}
c=
\begin{cases}
 \big(\frac{1-l^2}{4}+ \frac{l^2}{2}\ln(l)\big)^{-1} \qquad \quad \quad  \mbox{ if } m=2,\\
\big(\frac{1}{2m}+ \frac{l^m}{m(m-2)}-\frac{l^2}{2(m-2)}\big)^{-1} \quad \mbox{ if } m\ge3.
\end{cases}
\end{equation}
One verifies in a straightforward way that the resulting function $u \equiv u(\cdot; l)$ is decreasing for $l \le r \le 1$,
that $\theta(0) = 2/(m+2),$
and that $l \mapsto c \equiv c(l)$ is a continuous, strictly increasing function of $0 \le l < 1$.
We claim that $l \mapsto \theta(l)$ is also strictly increasing.
To verify that this is indeed the case, one could explicitly compute $\theta$ in terms of $l$, but the formula is rather complicated.
Instead we prove the claim by using the maximum principle. By contradiction, suppose there exist $0 \le l_2 < l_1 < 1$ with  $ \theta_2:=\theta(l_2) > \theta_1:= \theta(l_1)$.
By the considerations above,  $ c_2:= c(l_2) < c_1:= c(l_1).$ Hence $-\Delta(u_1 - u_2) = c_1 - c_2 > 0$ on $B_1 \setminus \ov B_{l_1}$
where $u_j:= u(\cdot, l_j)$ for $j=1,2$.
Since $\theta_1 <\theta_2$,
there exist $l_1 < r_1 < r_2 <1$ so that $(u_1 - u_2)(r) < 0$ for any $r_1 < r < r_2$,
contradicting the maximum principle.

From the formula \eqref{np2} of $u(\cdot; l)$ one infers that $\theta(l)$ is a continuous function of $l$
and that $\lim_{l \uparrow 1} \theta(l)= 1$. Hence for any $0 \le l < 1$, $u(\cdot; l)$ coincides with $u_\theta$
where $\theta \equiv \theta(l) = \fint_{B_1}  u(\cdot; l)$.
Altogether we have shown that $ \mathfrak{b}$ is a continuous, increasing bijection.

Define  $g : [0, 1) \to \R$ by
\begin{equation}\label{formula g}
g(l) :=f(\theta(l)) = (1 - \theta(l)) \int_{B_1} | \nabla u_{\theta(l)}|^2 \, , \qquad \theta(l) := \mathfrak{b}^{-1}(l) \, .
\end{equation}
In  view of \eqref{minimum} it then suffices to show that $g$ is increasing on $[0, 1)$.

 We first consider the case $m=2$. Integrating by parts, one obtains from \eqref{np1}
 $$
 \int _{B_1} |\nabla u_{\theta(l)}|^2 = c(l) 2\pi \int_l^1u_{\theta(l)}(r) r dr,
 $$
 and
 $$\theta(l)= l^2+ 2 \int_l^1u_{\theta(l)}(r) r dr.$$
Using \eqref{np2} and   \eqref{np3}, one infers from \eqref{formula g} that
 $$
 g(l)= 2\pi \frac{(\frac{1}{16}-\frac{l^2}{4}+ \frac{3l^4}{16}-\frac{l^4 \ln (l)}{4})(\frac{1}{8}-\frac{l^4}{8}+\frac{l^2 \ln(l)}{2} )}{(\frac 14-\frac{l^2}4+\frac{l^2}{2}\ln (l))^3},
 $$
 and a straightforward computation yields
 $$
 g'(l)= \pi  \frac{l(l^2-1)(-l^2+l^2\ln(l)+\ln(l)+1)(-5l^4+4l^4\ln(l)+4l^2+8l^2\ln(l)+1)}{(1-l^2+2l^2\ln(l))^4}.
 $$
By inspection one verifies that $g'(l) >0$ on $(0,1)$.

Without any additional effort we may consider the general case $m\ge 3$, and follow the line of arguments above. Integrating by parts, one has
  $$
  \int_{B_1} |\nabla u_{\theta(l)}|^2  = c(l) m \omega_m \int_l^1u_{\theta(l)}(r) r^{m-1} dr
  $$
  and one computes that
$$
\theta(l) = l^m+ m \int_l^1u_{\theta(l)}(r) r^{m-1} dr.
$$
Using formula \eqref{np2} for $m \ge 3$, one infers
$$\int_l^1u_{\theta(l)}(r)r^{m-1} dr= c(l) \, \Big (\frac{l^m(l^2-1)}{2m(m-2)} +\frac {l^{m+2}-1}{2m(m+2)}+ \frac{l^m(1-l^m)}{m^2(m-2)}+\frac{1-l^m}{2m^2}\Big)$$
so that by \eqref{formula g}
\begin{equation}\label{np4}
g(l)= c(l) m\omega_m \Big (1-l^m-m\int_l^1u_{\theta(l)}(r) r^{m-1} dr\Big)\int_l^1u_{\theta(l)}(r) r^{m-1} dr.
\end{equation}

In the case $m=3$, one gets in this way
$$g(l)=\frac{24\pi}{25}\frac{(5(1-l^3)+10l^3(1-l^3)-15l^3(1-l^2)-3(1-l^5))(1-l^5+5(l^3-l^2))}{(2l^3+1-3l^2)^3},$$
and a lengthy computation leads to the formula
$$g'(l)=\frac{24\pi}{25} \frac{2 l (20 l^4 + 67 l^3 + 84 l^2 + 46 l + 8)}{(2 l + 1)^4}.$$
Clearly,  $g'(l) >0$ on $(0,1)$ for $m=3$.\footnote{For general $m \ge 4$, the formula for $g'$ can be computed to be a quotient of two polynomials with degrees depending on $m$. We believe that $g'(l)$ is strictly positive for every $l$ on $(0,1)$, but a direct proof, covering all dimensions $m\ge 5$, based on the formula of $g'$ seems out of reach. For $m=4$ the quotient of the polynomials simplifies, and gives $g(l)=\omega_4\Big(\frac {16}{3}l^2+\frac{16}{9}\Big ).$ We see that for $m=4$,  $g(1)=\frac{64\omega_4}{9}$ agrees with the value $f(1)$ given in Remark \ref{rem1}.  We also have that $g(0)=\frac{16\omega_4}{3}$ agrees with the value $f(1/3)$ from Theorem \ref{the9}(i). Indeed for $m=4$ and $\theta=\frac13$ we have equality in \eqref{p22}.  Note that for $m=4$, $g(l)$ is increasing on $(0,1)$.}

\medskip

(d) In this last part we prove that equality in \eqref{p22} holds if and only if $\Omega$ is a ball and $u$ is a multiple of the torsion function for that ball.
Clearly, if  $\Omega$ is a ball and $u$ is a multiple of the torsion function for that ball, then  \eqref{p22} holds (see \eqref{minimum}).
Conversely, assume that equality holds  in \eqref{p22}. We re-scale the measure of $\Omega$ and the $L^{\infty}$ norm of $u$ as in the proof of Theorem \ref{the9}(i).
Equality in \eqref{p22} implies that $u$ has the same Dirichlet integral as its Schwarz rearrangement $u^*$,
\begin{equation}\label{Dirichlet}
\int_\Om|\nabla u|^2 = \int_{B_1}|\nabla u^*|^2 \, ,
\end{equation}
 and its Schwarz rearrangement is the solution
of the obstacle problem on the ball $B_1$ -- see \eqref{np1}. In view of the strict monotonicity of $f$ on $[\frac{2}{m+2}, \frac 23)$ (see part (a))
and the (strict) inequalities obtained above $\frac 23$ (see parts (a)-(c)), this implies that $\theta =\frac{2}{m+2}$, which corresponds to $l=0$ and to $c=2m$
(see \eqref{minimum}, \eqref{np3}).
It means that $u^*$ is a multiple of the torsion function on $B_1$ (see \eqref{np1}).

 In order to justify that $u$ has to be equal to $u^*$, recall that \eqref{Dirichlet} holds
  and that $u^*$, being a multiple of the torsion function on $B_1$, has a critical set of zero measure.
  Equality between $u$ and $u^*$, up to a translation, comes from the classical result of Brothers and Ziemer \cite[Theorem 1.1]{BZ88}.
  \hfill $\square$

 \medskip

\noindent
{\em Proof of Theorem \textup{\ref{the9}(ii).}}
The key ingredient into the proof is inequality \eqref{p22}.
First note that since for any $t > 0$, $\frac{1}{t^2}v_{\Omega}(tx)$ is the torsion function of $\frac{1}{t}\Omega$. Choosing
$t = (|\Omega| /\omega_m)^{1/m}$, one infers that
it suffices to prove estimate \eqref{p22a} in the case $|\Omega| = \omega_m$.

We apply (i) to
\begin{equation*}
u(x)=\frac{v_{\Omega}(x)}{\|v_{\Omega}\|_{\infty}},\qquad x\in \Omega.
\end{equation*}
 Observe that $\fint_{\Omega} u = \Phi(\Omega).$
First we consider the case where $\fint_{\Omega} u\ge \frac{2}{m+2}$. Then by \eqref{p25}
\begin{equation*}%\label{p59}
\frac{1}{\|v_{\Omega}\|_{\infty}^2}\big(1-\Phi(\Omega)\big)\int_{\Omega}|\nabla v_{\Omega}|^2\ge \frac{4m^2}{(m+2)^2}\omega_m\, .
\end{equation*}
Since $- \Delta v_\Omega = 1,$
\begin{equation*}%\label{p60}
\frac{1}{\omega_m\|v_{\Omega}\|_{\infty}}\int_{\Omega} | \nabla v_{\Omega}|^2 =
\frac{1}{\omega_m\|v_{\Omega}\|_{\infty}}\int_{\Omega} v_{\Omega}=\Phi(\Omega).
\end{equation*}
Since $\Phi(\Omega) \le 1$, we find that
\begin{align*}%\label{p61}
\|v_{\Omega}\|_{\infty}&\le \frac{(m+2)^2}{4m^2}\Phi(\Omega)\big(1-\Phi(\Omega)\big)\nonumber \\
& \le \frac{(m+2)^2}{4m^2}\big(1-\Phi(\Omega)\big),
\end{align*}
which gives \eqref{p22a}. Next consider the case $\fint_{\Omega} u\le \frac{2}{m+2}$.
Since by the de Saint-Venant's principle $\|v_\Omega\|_\infty \le \|v_{B_1}\|_\infty$ and since $ \|v_{B_1}\|_\infty = 1/2m$
 we find that
\begin{align*}%\label{p62}
1-\Phi(\Omega)&\ge \frac{m}{m+2}=\frac{2m^2}{m+2}.\frac{1}{2m}\nonumber \\ &
\ge \frac{2m^2}{m+2}\|v_{\Omega}\|_{\infty}.
\end{align*}
Note that $\frac{2m^2}{m+2} \ge \frac{4m^2}{(m+2)^2}$ and hence the estimate \eqref{p22a} also holds in this case.\\
Inequality \eqref{p22b} follows from \eqref{e4} and \eqref{p22a}.
\hfill $\square$

\medskip

\noindent
{\em  Proof of Theorem \textup{\ref{the9}(iii).}}
The key ingredient in the proof is inequality \eqref{p22}.
Since $\Omega$ is connected, $\lambda_1(\Omega)$ has multiplicity $1$ and hence both, $\varphi_{1,\Omega}$ and $E(\Omega)$,
are well defined.
First note that since for any $t > 0$, $t^{m/2}\varphi_{1, \Omega}(tx)$ is the positive $L^2$-normalised Dirichlet eigenfunction
of $\frac{1}{t}\Omega$, choosing
$t = (|\Omega| /\omega_m)^{1/m}$, one infers that
it suffices to prove  estimate \eqref{p22c} in the case $|\Omega| = \omega_m$.
By (26) in \cite{vdB}, one has
\begin{equation*}%\label{p64}
\|\varphi_{1,\Omega}\|_{\infty}\le \bigg(\frac{e}{2\pi m}\bigg)^{m/4}\lambda_1(\Omega)^{m/4}.
\end{equation*}
Hence $\varphi_{1,\Omega}\in L^{\infty}(\Omega)$.
We apply \eqref{p22} to
\begin{equation*}%\label{p65}
u(x)=\frac{\varphi_{1,\Omega}(x)}{\|\varphi_{1,\Omega}\|_{\infty}},\qquad x\in \Omega.
\end{equation*}
First we consider the case where $\fint_{\Omega} u\ge \frac{2}{m+2}$. Then by \eqref{p25}
\begin{equation*}%\label{p66}
\big(1-\fint_{\Omega}u\big) \int_{\Omega}|\nabla u|^2 \ge \frac{4m^2}{(m+2)^2}\omega_m .
\end{equation*}
Since $\int_{\Omega}|\nabla u|^2=\lambda_1(\Omega)\int_{\Omega}u^2\le \lambda_1(\Omega)\omega_m$, and $\fint_{\Omega} u=E(\Omega)$ we obtain,
\begin{equation}\label{p67}
\lambda_1(\Omega)\ge \frac{4m^2}{(m+2)^2}\big(1-E(\Omega)\big)^{-1}.
\end{equation}
Now let us consider the case where $\fint_{\Omega} u\le \frac{2}{m+2}$. Then $1-E(\Omega) \ge \frac{m}{m+2}$, and hence by Faber-Krahn,
\begin{equation}\label{p68}
\lambda_1(\Omega)\ge \lambda_1(B_1) \frac{m}{m+2}\big(1-E(\Omega)\big)^{-1}.
\end{equation}
Combining \eqref{p67} and \eqref{p68} gives
\begin{equation}\label{p69}
\lambda_1(\Omega)\ge \min\bigg\{\frac{4m^2}{(m+2)^2},\frac{m}{m+2}\lambda_1(B_1)\bigg\}\big(1-E(\Omega)\big)^{-1}.
\end{equation}
To finish the proof we recall that
\begin{equation}\label{p70}
\lambda_1(B_1)=j^2_{(m-2)/2}.
\end{equation}
By the results of \cite{LL}, we have that
\begin{equation}\label{p71}
j^2_{(m-2)/2}\ge \frac{m(m+8)}{4}.
\end{equation}
Hence by \eqref{p69}, \eqref{p70}, \eqref{p71},
\begin{equation*}%\label{p72}
\lambda_1(\Omega)\ge \frac{4m^2}{(m+2)^2}\big(1-E(\Omega)\big)^{-1} ,
\end{equation*}
which is inequality \eqref{p22c} in the case $|\Omega | = \omega_m$.
\hfill $\square$

\begin{remark}\label{rem1}
By an elementary computation, using the expression for $g$ in terms of $l$ from \eqref{np4}, one can show that
$ \lim_{\theta \uparrow 1} f(\theta) = \frac{4}{9}m^2 \omega_m,\,m\ge 2.$
\end{remark}

\medskip

Below we show that $\lambda_1(\Omega)$ cannot be bounded from above in terms of  $(1-E(\Omega))^{-1}|\Omega|^{-2/m}$ nor of $(1-\Phi(\Omega))^{-1}|\Omega|^{-2/m}$.
\begin{remark}\label{rem2} We have
\begin{equation}\label{ea1}
\sup\{\lambda_1(\Omega)(1-E(\Omega))|\Omega|^{2/m}:\Omega\,\textup{open, convex},\,0<|\Omega|<\infty\}=\infty,
\end{equation}
and
\begin{equation}\label{ea2}
\sup\{\lambda_1(\Omega)(1-\Phi(\Omega))|\Omega|^{2/m}:\Omega\,\textup{open, convex},\,0<|\Omega|<\infty\}=\infty.
\end{equation}
\end{remark}
\begin{proof}
To prove \eqref{ea1} we let $\Omega_n=(0,1)^{m-1}\times(0,n)$. Then $\lambda_1(\Omega_n)\ge (m-1)\pi^2.$ A straightforward calculation shows that for an interval of length $L,\, L>0,\,
E((0,L))=\frac{2}{\pi}$. By separation of variables $E(\Omega_n)=\frac{2^m}{\pi^m}$. We conclude that the supremum in \eqref{ea1} is bounded from below by $(m-1)\big(1-\frac{2^m}{\pi^m}\big)\pi^2n^{2/m}.$
Letting $n\rightarrow\infty$ concludes the proof.

To prove \eqref{ea2} we use \cite[Theorem 1.1 (i)]{DPGGLB} for $p=q=2$ to see that $\Phi(\Om_n)\le \frac23$. We conclude that the supremum in \eqref{ea2} is bounded from below by $\frac13(m-1)\pi^2n^{2/m}$. Letting $n\rightarrow\infty$ concludes the proof.
\end{proof}

\section{Proof of Theorem \ref{the4} \label{sec6}}

We start with the following observation.
\begin{lemma}\label{lem3}
Let $(\Omega_n)$ be a sequence of open sets in $\R^m$ with $0<|\Omega_n|<\infty,\,n\in \N$, $1\le p<\infty$, and let $f_n\in L^p(\Omega_n),\,n\in \N$ be a sequence of non-negative functions with $0<\|f_n\|_{\infty}<\infty$.  If
$(f_n)$ either localises in $L^p$ or $\kappa$-localises in $L^p$ then $(f_n)$ has vanishing mean to max ratio.
\end{lemma}
\begin{proof} Let $\varepsilon\in(0,1)$ be arbitrary. By hypothesis there exists a sequence $(A_n)$ satisfying \eqref{ea12}. Then for all $n$ sufficiently large $|A_n|/|\Omega_n|<\varepsilon$,
and $\|f_n\|^p_p\le \kappa^{-1}(1-\varepsilon)^{-1}\int_{A_n}f^p_n$. Then for all such $n$,
\begin{equation}\label{e60}
\|f_n\|^p_p\le \kappa^{-1}(1-\varepsilon)^{-1}\int_{A_n}f^p_n\le \kappa^{-1}(1-\varepsilon)^{-1}\|f_n\|^p_{\infty}|A_n|\le \frac{\varepsilon\kappa^{-1}}{1-\varepsilon}\|f_n\|^p_{\infty}|\Om_n|.
\end{equation}
By H\"older's inequality,
\begin{equation}\label{e60b}
\bigg(\int_{\Om_n}f_n\bigg)^p\le \|f_n\|_p^p |\Om_n|^{p-1}.
\end{equation}
By \eqref{e60} and \eqref{e60b} we have for all $n$ sufficiently large,
\begin{equation*}%\label{e60a}
\frac{\|f_n\|_1}{|\Om_n|\|f_n\|_{\infty}}\le \bigg(\frac{\varepsilon\kappa^{-1}}{1-\varepsilon}\bigg)^{1/p}.
\end{equation*}
Since $\varepsilon\in(0,1)$ was arbitrary, $(f_n)$ has vanishing mean to max ratio.
\end{proof}

\noindent{\it Proof of Theorem} \ref{the4}. We obtain by
 \eqref{e24} and \eqref{e41}
\begin{equation*}%\label{e65}
v_{\Omega,V}(x)=\sum_{j =1}^{\infty} \lambda_j(\Omega,V)^{-1} \bigg(\int_{\Omega}\varphi_{j, \Omega,V}\bigg)\varphi_{j, \Omega,V}(x).
\end{equation*}
Integrating with respect to $x$ over $\Omega$ gives
\begin{align*}%\label{e66}
\int_{\Omega}v_{\Omega,V}&=\sum_{j =1}^{\infty} \lambda_j(\Omega,V)^{-1} \bigg(\int_{\Omega}\varphi_{j, \Omega,V}\bigg)^2\nonumber \\ &
\ge \lambda_1(\Omega,V)^{-1} \bigg(\int_{\Omega}\varphi_{1, \Omega,V}\bigg)^2.
\end{align*}
Multiplying both sides with $\lambda_1(\Omega,V)$, and using the definition of $\mathfrak{d}_m$ in \eqref{e29a} gives
\begin{equation*}%\label{e67}
\mathfrak{d}_m\frac{\|v_{\Omega,V}\|_1}{\|v_{\Omega,V}\|_{\infty}}\ge \bigg(\int_{\Omega}\varphi_{1, \Omega,V}\bigg)^2.
\end{equation*}
This implies that
\begin{equation}\label{e68}
\mathfrak{d}_m\Phi(\Omega,V)\ge \frac{1}{|\Omega|}\bigg(\int_{\Omega}\varphi_{1, \Omega,V}\bigg)^2.
\end{equation}
Suppose $(v_{\Om_n,V_n})$ either localises or $\kappa$- localises in $L^1$. By Lemma \ref{lem3} for $p=1$,  $\lim_{n\rightarrow\infty}\Phi(\Omega_n,V_n)=0$. By \eqref{e68}, $\lim_{n\rightarrow\infty}\frac{1}{|\Omega_n|}\big(\int_{\Omega_n}\varphi_{1, \Omega_n,V_n}\big)^2=0$.
This implies localisation of $(\varphi_{1, \Omega_n,V_n})$ in $L^2$ by Lemma 3 in \cite{vdBBDPG}, and vanishing efficiency by Lemma \ref{lem3} for $p=2$.
\hfill  $\square$

Theorem \ref{the4} implies that if $(\Omega_n,V_n)$ satisfies the $\eta$ condition of Theorem \ref{the1}(ii), and if either $(v_{\Omega_n})$ or $(v_{\Omega_n,V_n})$ have non-vanishing efficiencies then both $(v_{\Omega_n})$ and $(v_{\Omega_n,V_n})$ are not localising.

\bigskip

\noindent
{\bf Acknowledgments.} MvdB and TK acknowledge support by the Leverhulme Trust through Emeritus Fellowship EM-2018-011-9,
and the Swiss National Science Foundation respectively.  DB was supported by the LabEx PER\-SYVAL-Lab GeoSpec (ANR-11-LABX-0025-01) and ANR SHAPO (ANR-18-CE40-0013).

\end{document}